\documentclass[10pt,reqno]{amsart}

\usepackage[a4paper, margin=1.5in]{geometry}
\usepackage {color}
\usepackage{amsmath,amsthm,amssymb}
\usepackage{epsfig}
\usepackage{graphicx}
\usepackage{amssymb}
\usepackage{latexsym}
\usepackage{pdfpages}

\usepackage{pgf}

\usepackage{pgfplots}

\usepackage{ulem} 

\usepackage{ifpdf}
\newcommand{\res}{\!\!\mathop{\hbox{
                                \vrule height 7pt width .5pt depth 0pt
                                \vrule height .5pt width 6pt depth 0pt}}
                                \nolimits}

\def\z{{\bf z}}

\def\divi{\hbox{\rm div\,}}

\ifpdf 
  \usepackage[hidelinks]{hyperref}
\else 
\fi 

 \usepackage[usenames,dvipsnames]{pstricks}
 \usepackage{epsfig}
 \usepackage{pst-grad} 
 \usepackage{pst-plot} 

\newtheorem{theorem}{Theorem}[section]
\newtheorem{lemma}[theorem]{Lemma}
\newtheorem{definition}[theorem]{Definition}
\newtheorem{proposition}[theorem]{Proposition}
\newtheorem{corollary}[theorem]{Corollary}
\newtheorem{remark}[theorem]{Remark}
\newtheorem{example}[theorem]{Example}

\newtheorem*{theorem*}{\it Theorem}

\def\divi{\hbox{\rm div\,}}
\def\R{\mathbb R}
\def\N{\mathbb N}

\numberwithin{equation}{section}

\def\1{\raisebox{2pt}{\rm{$\chi$}}}

\def\z{{\bf z}}

\definecolor{violet(ryb)}{rgb}{0.53, 0.0, 0.69}
\usepackage{mathtools}
\mathtoolsset{showonlyrefs}

\begin{document}

\title[]{Evolution problem for the $1$-Laplacian with mixed  boundary conditions}

\author{N. Igbida}
\address{N. Igbida: Institut de recherche XLIM-DMI, UMR-CNRS 6172, Universit\'{e} de Limoges, France. {\tt noureddine.igbida@unilim.fr}}

\author{J. M. Maz\'{o}n}
\address{J. M. Maz\'{o}n: Departament d'An\`{a}lisi Matem\`atica,
Universitat de Val\`encia, Valencia, Spain. {\tt mazon@uv.es }}

\author{J. Toledo}
\address{J. Toledo: Departament d'An\`{a}lisi Matem\`atica,
Universitat de Val\`encia, Valencia, Spain. {\tt toledojj@uv.es }}

\keywords{Total variation flow, $1$-Laplacian, mixed boundary conditions, maximal monotone operators.\\
\indent 2010 {\it Mathematics Subject Classification.}
35K92,47H06, 47H20,47J35.}

\begin{abstract}  This paper deals with evolution problem for the $1$-Laplacian with mixed  boundary conditions on a bounded open set $\Omega$ of $\R^N$. We prove existence and uniqueness of  strong solutions for data in $L^2(\Omega)$ by mean of the theory of maximal monotone operator. We   also see that if the flux on the boundary is~$1$   (that is, the maximum possible) then these strong solutions can be seen as the large solutions introduced in \cite{MP}. We give explicit examples of solutions.
\end{abstract}

\maketitle

\begin{center}\it In memory of our friend and collaborator Fr\'{e}d\'{e}rique Simondon
\end{center}
\bigskip


\section{Introduction}

The goal of this paper is to established the well-posedness of the problem
\begin{equation}\label{ProblemI}
    \left\{
    \begin{array}{ll}
        u_t -  \Delta_1 u\ni 0   &\hbox{in} \ \ (0,T) \times\Omega,
        \\[10pt] \frac{Du}{\vert Du \vert} \cdot\nu = \psi \quad &\hbox{on} \ \ (0,T) \times\Gamma_N,
        \\[10pt] u = g   &\hbox{on} \ \ (0,T) \times\Gamma_D,
         \\[10pt] u(0) = u_0,
    \end{array}
    \right.
\end{equation}
where $\Omega \subset \R^N$ is a bounded domain with smooth boundary $\partial \Omega$ satisfying
$$\partial \Omega = \Gamma_D \cup \Gamma_N,$$
where $\Gamma_D$ and $\Gamma_N$ are assumed to be disjoint, $\psi \in L^\infty(\Gamma_N)$, $||\psi||_\infty\le 1$, $g \in L^1(\Gamma_D)$ and $u_0 \in L^2(\Omega)$.

In the case $\Gamma_N = \emptyset$ , problem~\eqref{ProblemI} corresponds to the Dirichlet problem
\begin{equation}\label{DirichletProblem}
    \left\{
    \begin{array}{ll}
     u_t - \Delta_1 u \ni 0 &\hbox{in} \ \ (0,T) \times\Omega,
        \\[10pt] u = g   &\hbox{on} \ \ (0,T) \times\partial \Omega,
         \\[10pt] u(0) = u_0,
    \end{array}
    \right.
\end{equation}
 that was studied in \cite{ABCM2} (see also \cite{ACMBook}).   The motivation to study such problem comes from a  variational approach for filling in regions of missing data in digital images  introduced in \cite{BBCS}. Now, the study of the elliptic Dirichlet problem for the $1$-Laplacian starts with the paper \cite{SWZ} with the study of the least gradient problem, see the monograph \cite{GMBook} for the state of the art of this problem.

 The case $\Gamma_D = \emptyset$  corresponds to the Neumann problem
\begin{equation}\label{NeumannProblem}
    \left\{
    \begin{array}{ll}
      u_t -  \Delta_1 u\ni 0   &\hbox{in} \ \ (0,T) \times\Omega,
        \\[10pt] \frac{Du}{\vert Du \vert} \cdot \nu= \psi \quad &\hbox{on} \ \ (0,T) \times \partial \Omega,
         \\[10pt] u(0) = u_0,
    \end{array}
    \right.
\end{equation}
which was studied, for the homogeneous case, i.e., $\psi =0$, in \cite{ABCM1} (see also \cite{ACMBook}),   and whose motivation was the ROF-model in image restoration  introduced in \cite{ROF}. For the  nonhomogeneous case,  with $\Vert \psi \Vert_\infty <1$, its associated elliptic problem posed in $L^2(\Omega)$
was studied in~\cite{BCN},   but let us point out that, to our knowledge, the results on the non-homogeneous Neumann evolution problem are new.  The case $\psi=1$ was studied in~\cite{MP} to understand large solutions (see Section~\ref{sectls} later on).

It is clear that in order to have solutions to problem \eqref{ProblemI} we need to impose  the restriction $\Vert \psi \Vert_\infty \leq 1$. As we will see, under the restriction $\Vert \psi \Vert_\infty <1$,  the problem is the gradient flow in $L^2(\Omega)$ of a convex and lower semi-continuous functional, and consequently in this case we get existence and uniqueness of strong solution for all initial data in $L^2(\Omega)$. In the case  $\Vert \psi \Vert_\infty = 1$ we first prove existence and uniqueness of mild solutions, afterwards we see that they are, in fact, strong solutions.

To conclude this introduction, let us mention here that our proposed solution concept is natural. We first build the solution for the standard Euler implicit discretization of problem~\eqref{ProblemI} by minimizing an energy functional in the $BV$ space, a method typically used in total variation problems with mixed boundary conditions. Subsequently, Fenchel-Rockafellar duality allows us to deduce the solution notion and the PDE linked to such stationary problem. We finally obtain the solution of problem~\eqref{ProblemI} by means of  nonlinear semigroup theory.

 The paper is organized as follows: in Section \ref{Preli}  we introduce the results we need about
functions of bounded variation  and the  Anzellotti Green's formula.    In Section \ref{lasec3} we establish the main results. Section \ref{sectionla4} deals with  proofs for the case where $ \Vert \psi\Vert_\infty <1,$ and Section \ref{sectionla5} is dedicated to the  proofs for  the case general case $||\psi||_\infty \leq  1$. In Section~\ref{sectls}   we study the relation  with large solutions. Finally in Section \ref{ExplicitS}  we compute explicit  solutions.   In an Appendix section we collect the results we use from Nonlinear Semigroup Theory.

\section{Preliminaries on $BV$ functions and Anzellotti pairings}\label{Preli}

Due to the linear growth condition on the Lagrangian, the natural energy
space to study the problem is the space of functions of bounded variation. Let us recall several facts concerning functions of bounded variation (for further information we refer to  \cite{AFP}). Throughout the whole paper, we assume that $\Omega \subset \R^N$ is an open bounded set with { $C^{1,1}$~boundary.

\begin{definition}{\rm
		A function $u \in L^1(\Omega)$ whose partial derivatives in the sense
		of distributions are measures with finite total variation in $\Omega$
		is called a function of bounded variation.
		The space of such functions will be denoted by $BV(\Omega)$. In other words,
		$u \in BV(\Omega)$ if and only if there exist Radon measures
		$\mu_{1},\ldots,\mu_{N}$ defined in $\Omega$ with finite total mass
		in $\Omega$ and
		\begin{equation}
			\int_{\Omega} u \, D_{i} \varphi \, dx = -\int_{\Omega} \varphi \, d\mu_{i}
		\end{equation}
		for all $\varphi \in C^{\infty}_{0}(\Omega)$,
		$i=1,\ldots,N$. Thus, the distributional gradient of $u$ (denoted $Du$) is a vector valued measure with finite total variation
		\begin{equation}
			\label{deftv}
			\vert  Du \vert (\Omega) = \sup \bigg\{ \int_{\Omega} \, u \, \mathrm{div} \, \varphi \, dx: \,\, \varphi \in C^{\infty}_{0}(\Omega; \R^N), \,\, |\varphi(x)| \leq 1
			\, \, \hbox{for $x \in \Omega$} \bigg\}.
		\end{equation}
		The space $BV(\Omega)$ is endowed with the norm
		\begin{equation}
			\| u \|_{BV(\Omega)} = \| u \|_{L^1(\Omega)} +
			\vert  Du \vert (\Omega).
		\end{equation}
	}
\end{definition}

\begin{definition}{\rm Let $u, u_n \in BV(\Omega)$. We say that $\{u_n \}$ {\it strictly converges} in $BV(\Omega)$ to $u$ if $\{u_n \}$ converges to $u$ in $L^1(\Omega)$ and $\vert Du_n \vert (\Omega)$ converges to $\vert Du \vert (\Omega)$ as $n \to \infty$.
	}
\end{definition}

It is well-known (see \cite{AFP}) the following result about the existence of the trace on the boundary of functions of bounded variations.

\begin{theorem}\label{Trace} Let  $\Omega \subset \R^N$ be an open bounded set with Lipschitz boundary and $u \in BV(\Omega)$. Then, for $\mathcal{H}^{N-1}$-almost every $x \in \partial \Omega$ there exists $u^\Omega(x) \in \R$ such that
	$$\lim_{\rho \to 0} \frac{1}{\rho^N} \int_{\Omega \cap B_\rho(x)} \vert u(y) - u^\Omega(x) \vert dy =0.$$
	Moreover, $\Vert u^\Omega \Vert_{L^1(\partial (\Omega)} \leq C \Vert u \Vert_{BV}$ for some constant $C$ depending only on $\Omega$, the extension $\overline{u}$ of $u$ to $0$ out of $\Omega$ belongs to $BV(\R^N)$, and
	$$D \overline{u} = Du + u^\Omega \, \mathcal{H}^{N-1} \res \partial \Omega.$$
	The trace operator $u \mapsto u^\Omega$ is a continuous bijection between $BV(\Omega)$, endowed with the topology induced by the strict convergence, and  $L^1(\partial \Omega, \mathcal{H}^{N-1} \res \partial \Omega)$.
\end{theorem}

We will denote the trace operator by $T_r$,  and when there is no confusion we will denote  by $u$ to the trace  $Tr(u)$.

We recall the following embedding theorem  stated in \cite[Theorem
6.5.7]{Mazya}.

\begin{theorem}\label{mazinq}
	Suppose that $\Omega \subset \R^N$  ($N\ge 2$) is an open bounded set with
	Lipchitz boundary.  Then,   there exists  constants $C_1, C_2>0$ such that
	\begin{equation}
		\label{eq:7N}
		\Vert u \Vert_{L^{\frac{N}{N-1}}(\Omega)} \le C_1
		\vert Du \vert (\Omega)
		+ C_2 \Vert u\Vert_{L^1(\partial \Omega)},
	\end{equation}
	for every $u \in BV(\Omega)$.
\end{theorem}
  For   $N=1$ and $\Omega=]a,b[$, we have that $BV(a,b)\subset L^\infty(a,b)$, and for any $x,y\in ]a,b[$,
$$|f(x)|\le \vert Du \vert (]a,b[)+|f(y)|.$$

We also have (see~\cite{Sabina}) that:

\begin{theorem}\label{pepesab}
	Suppose that $\Omega \subset \R^N$ is an open bounded set with
	$C^{1,1}$ boundary.  Then,   there exists a constant $C_\Omega>0$ such that
	\begin{equation}
		\label{eq:7}
		\Vert u \Vert_{L^{1}(\partial\Omega)} \le       \vert Du \vert (\Omega)
		+ C_\Omega\Vert  u \Vert_{L^1(\Omega)}
	\end{equation}
	for every $u \in BV(\Omega)$.
\end{theorem}

Modica in  \cite[Proposition~1.2]{Modica} gives the following result.

\begin{proposition}\label{ModicaP} Let $\tau : \partial \Omega \times \R \rightarrow \R$ be a Borel function, and  for $u \in BV(\Omega)$ let
	$$F(u):= \int_\Omega \vert Du \vert + \int_{\partial \Omega} \tau(x,  u(x)) \, \mathcal{H}^{N-1}(x).$$
	If
	$$\vert \tau(x,s_1) - \tau(x,s_2) \vert \leq \vert s_1 - s_2 \vert \quad \hbox{for $\mathcal{H}^{N-1}$-a.e. } x \in \partial \Omega, \  \hbox{and for all }  s_1, s_2 \in \R,$$
	then the functional $F$ is lower semi-continuous on $BV(\Omega)$ with respect to the topology of $L^1(\Omega)$.
\end{proposition}

In the proof of the above result Modica uses the inequality~\eqref{eq:7},  and says that it is  true, with the constant $1$ in front of $\vert Du \vert (\Omega)$, when  $\partial \Omega$ is smooth enough. In~\cite{Sabina} it  is shown that this is true if $\partial \Omega$ is $C^{1,1}$ but it is not true, in general, under less regularity of the boundary.

We   now state several results from \cite{Anzellotti1} that we use afterwards. Let, for $1\le p<+\infty$,
\begin{equation}
	\label{xw}
	X_p(\Omega) = \{ \z \in L^{\infty}(\Omega; \R^N):
	\divi(\z)\in L^p(\Omega) \}.
\end{equation}

\begin{definition}{\rm
		For $\z \in X_p(\Omega)$ and $u \in BV(\Omega) \cap L^{p'}(\Omega)$,
		define the functional $(\z,Du): C^{\infty}_{0}(\Omega) \rightarrow \R$
		by the formula
		\begin{equation}\label{Anzel1}
			\langle (\z,Du),\varphi \rangle = - \int_{\Omega} u \, \varphi \, \divi(\z) \, dx -
			\int_{\Omega} u \, \z \cdot \nabla \varphi \, dx.
		\end{equation}
	}
\end{definition}

The following result collects some of the most important properties of the pairing $(\z, Du)$, formally defined only as a distribution on $\Omega$.

\begin{proposition}\label{Anzelotti}
	The distribution $(\z,Du)$ is a Radon measure in $\Omega$. Moreover,
	\begin{equation}\label{eq:absolutecontinuity}
		\bigg\vert \int_{B} (\z,Du) \bigg\vert \leq \int_{B} |(\z,Du)| \leq \| \z
		\|_{\infty}
		\int_{B} \vert Du \vert
	\end{equation}
	for any Borel set $B \subseteq \Omega$. In particular, $(\z,Du)$ is absolutely continuous with respect to $\vert Du \vert$. Furthermore,
	\begin{equation}
		\int_{\Omega} (\z,Dw) = \int_{\Omega} \z \cdot \nabla w \, dx \ \ \ \forall\,   w
		\in W^{1,1}(\Omega) \cap L^{\infty}(\Omega),
	\end{equation}
	 with what $(\z,Du)$ agrees on Sobolev functions with the dot product of $\z$ and $\nabla u$.
\end{proposition}

By \eqref{eq:absolutecontinuity}, the measure $(\z,Du)$ has a Radon-Nikodym derivative with respect to $|Du|$
$$\theta (\z,Du, \cdot):= \frac{d[(\z, Du)]}{d|Du|},$$
which is a $\vert Du \vert$-measurable function from $\Omega$ to $\R$ such that
\begin{equation}\label{410}
	\int_{B} (\z,Du) = \int_{B} \theta (\z,Du,x) \vert Du \vert
\end{equation}
for any Borel set $B \subseteq \Omega$. We have that
\begin{equation}
	\| \theta (\z, Du, \cdot) \|_{L^{\infty}(\Omega,\vert Du \vert )} \leq
	\| \z \|_{L^{\infty}(\Omega;\R^N)}.
\end{equation}

Moreover, the
following  chain rule for $(\z,D(\cdot))$ holds.

\begin{proposition}\label{prop:chain-rule}
	Let $\Omega$ be a bounded domain with a Lipschitz-continuous boundary
	$\partial\Omega$ and for $1\le p\le N$ and $p^{\mbox{}_{\prime}}$ given by
	$1=\tfrac{1}{p}+\tfrac{1}{p^{\prime}}$, let $\z \in X_p(\Omega)$ and
	$w \in BV(\Omega)_{p^{\mbox{}_{\prime}}}$. Then, for every Lipschitz continuous,
	monotonically  non-decreasing function $T : \R \rightarrow \R$, one has that
	\begin{equation}\label{E1paring12}
		\theta(\z, D(T \circ w),x) = \theta (\z, Dw, x)\ \ \text{for
			$\vert Dw \vert$-a.e. $x\in \Omega$.}
	\end{equation}
\end{proposition}

In \cite{Anzellotti1}, a weak trace on $\partial \Omega$ of the normal component
of   $\z \in X_p(\Omega)$ is also defined. Concretely, it is proved that there exists a linear operator
$[\z, \nu_\Omega] : X(\Omega) \rightarrow L^{\infty}(\partial \Omega)$ such that
$$\Vert [\z, \nu_\Omega] \Vert_{\infty} \leq \Vert \z \Vert_{\infty},$$
$$[\z, \nu_\Omega] (x) = \z(x) \cdot \nu(x) \   \ {\rm for \ all} \ x \in
\partial
\Omega \   {\rm if}
\ \z \in C^1(\overline{\Omega}, \R^N),$$
being $\nu_\Omega(x)$ the unit outward normal on $x\in\partial \Omega$.  Moreover, the following {\it Green's formula}, relating the function $[\z, \nu_\Omega]$ and the measure $(\z, Dw)$, was proved in the same paper.

\begin{theorem}
	For all $\z \in X_p(\Omega)$ and $u \in BV(\Omega) \cap L^{p'}(\Omega)$, we have
	\begin{equation}\label{Green}
		\int_{\Omega} u \ \divi (\z) \, dx + \int_{\Omega} (\z, Du) = \int_{\partial \Omega} u \, [\z, \nu_\Omega] \, d\mathcal{H}^{N-1}.
	\end{equation}
\end{theorem}

\section{Main results}\label{lasec3}

To address problem \eqref{ProblemI}  we begin by examining the associated stationary problem which corresponds to the standard Euler implicit discretization. For a given $f\in L^2(\Omega),$ we consider
\begin{equation} \label{ProblemIstat}
	\left\{
	\begin{array}{ll}
		 u - \Delta_1 u  \ni f  &\hbox{in} \ \  \Omega,
		\\[10pt] \frac{Du}{\vert Du \vert} \cdot\nu = \psi \quad &\hbox{on} \ \  \Gamma_N,
		\\[10pt] u = g   &\hbox{on} \ \  \Gamma_D
	\end{array}
	\right.
\end{equation}
This problem inherently requires a necessary condition for existence, directly related to the constraint  $\left\Vert  \frac{Du}{\vert Du \vert} \right\Vert_\infty  \leq 1,$ which is expressed as
\begin{equation} \label{psicond}
	\Vert \psi\Vert_\infty \leq 1.
\end{equation}
Furthermore, it is established that the Dirichlet boundary condition   $u=g$ is often unsuitable for this class of problem.  Specifically,    solutions satisfying the boundary data in the sense of trace  typically do not exist.  A concrete illustration is provided  in~\cite[Example~5.25]{GMBook} by the following example in $\Omega=B(0,1),$
\begin{equation}\label{pd1001} \left\{
\begin{array}{ll}
	 -  \Delta_1 u \ni 0   &\hbox{in} \ \  \Omega ,
	\\[6pt] u = g   &\hbox{on} \ \ \partial\Omega,
\end{array}
\right.
\end{equation}
with
$$g=\chi_{F_\infty} $$
where $F_\infty\subset \partial\Omega.$ More precisely, it is proven in \cite[Theorem~5.24]{GMBook} that the   optimization problem
$$\min\left\{\int_\Omega|Du|:u\in BV(\Omega),\ u=g\ \hbox{on }\partial\Omega\right\}$$
has no solution.    To solve this     difficulty, in \cite{ABCM2}  (see also \cite{Giaetal},  \cite{GMBook})  proves definitely that the natural way to solve~\eqref{pd1001} is given by the {\it relaxed}   optimization problem
	\begin{equation}\min\left\{ 	\int_\Omega|Du|+\int_{\partial\Omega}|u-g|d\mathcal{H}^{N-1}\: : \quad u\in BV(\Omega)\right\},\end{equation}
	which implies in turn   the {\it relaxed} condition
$$	[\z,\nu_\Omega] \in \hbox{sign}(g-u) \quad  \hbox{in} \ \ \Gamma_D,$$
  where the vector field $\z$ is a realization of $\frac{Du}{\vert Du \vert}$.
   With this relaxed  condition    we will able to  prove existence and uniqueness of solutions.

 \medskip
  So in order to solve the problem \eqref{ProblemIstat},
 we aim to minimize in $L^2(\Omega)$ the energy functional
 $\Phi_f\: :\:  L^2(\Omega) \rightarrow ]-\infty, + \infty]$ defined by
 $$\Phi_f(v):=  \mathcal{F}_{\psi,g} (v)+  \frac12 \int_{ \Omega}(v- f)^2 \, dx,$$
 where
  $$\mathcal{F}_{\psi,g}(v):= \left\{ \begin{array}{ll} \displaystyle \int_\Omega \vert Dv \vert -\int_{\Gamma_N} \psi v  \, d  \mathcal{H}^{N-1} + \int_{\Gamma_D} \vert g - v \vert  \, d  \mathcal{H}^{N-1}   &\hbox{if} \ v \in BV(\Omega) \cap L^2(\Omega)   \\[10pt] +\infty &\hbox{if} \ v \in L^2(\Omega)  \setminus BV(\Omega). \end{array}\right.$$

 \medskip As we see above the  condition $\Vert \psi \Vert_\infty \leq 1$ is natural and not restrictive.  However,  as we will see, we need to consider separately the case $\Vert \psi \Vert_\infty < 1$ due to the fact that in this case the associated energy functional to the problem is lower semi-continuous  in $L^2(\Omega)$, but this does not happen when $\Vert \psi \Vert_\infty =1 $.
By means of an example, we show that the condition $\Vert \psi \Vert_\infty = 1$ leads to inconsistencies in the optimization problem. To demonstrate this, consider the functional
  		$$F(u):= \left\{\begin{array}{ll} \displaystyle\int_\Omega \vert Du \vert - \int_{\partial \Omega} Tr(u) d \mathcal{H}^1 \quad &\hbox{if} \ u \in BV(\Omega) \\[10pt] + \infty \quad &\hbox{if} \ u \in  L^2(\Omega) \setminus BV(\Omega), \end{array}  \right., $$
  		being $\Omega = B_1(0)$ in $\R^2$. Let $u, u_n: \Omega \rightarrow \R$ be the functions
  		$$u(x):= \frac{1}{(1 - \Vert x \Vert)^{\frac14}}, \quad u \in L^2(\Omega) \setminus BV(\Omega),$$
  		$$u_n(x):= \left\{\begin{array}{ll} \displaystyle\frac{1}{(1 - \Vert x \Vert)^{\frac14}}, \quad &\Vert x \Vert < 1- \frac{1}{n} \\[14pt] n^{\frac14} \quad & 1- \frac{1}{n} \leq \Vert x \Vert < 1. \end{array}  \right.$$
  		We have
  		$$\int_\Omega \vert Du_n \vert = 2 \pi n^{\frac14}, \quad \int_{\partial \Omega} u_n d \mathcal{H}^1 = 2\pi \left(n^{\frac14} + \frac{1}{3 n^{\frac34}} - \frac43 \right).$$
  		Then, $F(u) = + \infty$ and
  		$$\liminf_{n \to \infty} F(u_n) = - \frac{8 \pi}{3}.$$
  		Hence, since $u_n \to u$ in $L^2(\Omega)$, we have that $F$ is not lower semi-continuous.

    \bigskip
    In the case   $\Vert  \psi \Vert_\infty < 1,$    we can handle the optimization problem using standard techniques from the calculus of variations.

 \begin{lemma}\label{conlower} If  $\psi \in L^\infty(\Gamma_N)$ is  such that  $\Vert  \psi \Vert_\infty < 1$ and  $g \in L^1(\Gamma_D)$, then functional $\mathcal{F}_{\psi,g}$ is convex and lower semi-continuous in $L^2(\Omega)$.
 \end{lemma}
 \begin{proof} Obviously  $\mathcal{F}_{\psi,g}$ is convex. Let us see that is lower semi-continuous in $L^2(\Omega)$.  Indeed, let $u_n\in BV(\Omega) \cap L^2(\Omega)$ be such that $u_n\to u$ in $L^2(\Omega)$. Then, if $u\in BV(\Omega)$, by Proposition~\ref{ModicaP} we have that $$\mathcal{F}_{\psi,g}(u)\le \liminf_n\mathcal{F}_\psi(u_n).$$ Now, if $u\notin BV(\Omega)$, let us see that $\liminf_n\mathcal{F}_{\psi,g}(u_n)=+\infty$: on the contrary, there exists $M >0$ such that
 	$$\begin{array}{c}
 		\displaystyle\int_\Omega \vert Du_n \vert\le M +||\psi||_\infty\int_{\Gamma_N}|u_n| \, d  \mathcal{H}^{N-1} - \int_{\Gamma_D} \vert g -  u_n \vert  \, d  \mathcal{H}^{N-1} \\[12pt] \displaystyle \leq M +||\psi||_\infty\int_{\partial \Omega}|u_n|  \, d  \mathcal{H}^{N-1}\\[12pt]
 		\displaystyle \le M+||\psi||_\infty\int_{\Omega} \vert Du_n \vert+C_\Omega||\psi||_\infty \int_\Omega \vert u_n \vert,
 	\end{array}$$
 	where Theorem~\ref{pepesab} has been used. Hence,
 	$$(1-||\psi||_\infty)\int_\Omega \vert Du_n \vert\le  M+C_\Omega||\psi||_\infty \int_\Omega \vert u_n \vert \quad\forall \, n \in \N.$$
 	Therefore, $\{ u_n\}_n$ is bounded in $BV(\Omega)$,  and, since $u_n\to u$ in $L^2$, we have that $u\in BV(\Omega)$, which gives a contradiction.
 \end{proof}

 These two key observations require us to study $\psi$ in two separate settings: one for $\Vert \psi \Vert_\infty < 1$ 	and another for $\Vert \psi \Vert_\infty =1,$  using varied yet related techniques. Significantly, the second setting produces remarkable  outcomes concerning the concept of large solutions, characterized by a trace that blows up at the boundary.  We will now present the main results whose proofs are addressed to sections~\ref{sectionla4}~, \ref{sectionla5} and~\ref{sectls}.

 \subsection{The case where $\Vert \psi \Vert_\infty < 1$}

  Our first main result concerns existence and uniqueness of a solution to the stationary problem \eqref{ProblemIstat} and its connection with the    optimization problem associated with $\Phi_f.$

  \begin{theorem}\label{ThStat1}
 Let $f\in L^2(\Omega),$  $g\in L^1(\Gamma_D)$ and $\psi\in L^\infty(\Gamma_D)$ satisfying
 \begin{equation}\label{psicondstrict}
 		\Vert \psi\Vert_\infty < 1.
 \end{equation}
 Then,  the  problem $$\displaystyle \min_{u\in L^2(\Omega)}\Phi_f(u) $$ or
 \begin{equation}\label{ResOptim}
  (P):=\min_{ u\in BV(\Omega) \cap L^2(\Omega)}\left\{  \int_\Omega \vert Dv \vert -\int_{\Gamma_N} \psi v  \, d  \mathcal{H}^{N-1} + \int_{\Gamma_D} \vert g - v \vert  \, d  \mathcal{H}^{N-1}  + \frac{1}{2} \int_{ \Omega}(v- f)^2 \, dx  \right\}
 \end{equation}
  has a unique solution $u\in BV(\Omega)\cap L^2(\Omega)$.  Moreover,
\begin{enumerate}
\item      The following duality holds:
$$(P)=(M),$$
where
\begin{equation}\label{ResOptimD}
\begin{array}{r}
\displaystyle  (M): =\max \left\{ \frac{1}{2} \int_{ \Omega} f^2  \, dx - \frac{1}{2} \int_{ \Omega} \xi^2  \, dx  +\int_{\Gamma_D}  [z,\nu_\Omega]g \, d  \mathcal{H}^{N-1}  :   \quad   \right. \\  \\
   \xi\in L^2(\Omega),\: z\in L^\infty(\Omega)^N, \Vert z\Vert_\infty  \leq 1,\quad \\  \\
   -\divi   z =f- \xi \hbox{ in }\Omega,\:  [z,\nu_\Omega] = \psi   \hbox{ in  } \Gamma_N
   \Big\},
 	\end{array}
 \end{equation}
 and $(M)$ is attained.
 	\item If $u$ is a solution (i.e., a minimizer) of $(P)$ and $(\z,\xi)$ is a solution of $(M)$, then $\xi=u$ and the couple $(u,\z)$ solves the  PDE problem \eqref{ProblemIstat} in the following sense
 	  \begin{equation}\label{StatPDE}
 	 	\left\{ \begin{array}{llll}  u - \divi \z = f  \quad &\hbox{in} \ \ \mathcal{D}'(\Omega),  \\[10pt] (\z,Du) = \vert Du \vert \quad &\hbox{as measures},\\[10pt] [\z,\nu_\Omega] = \psi \quad &\hbox{in} \ \ \Gamma_N, \\[10pt] [\z,\nu_\Omega] \in \hbox{sign}(g-u) \quad &\hbox{in} \ \ \Gamma_D.   \end{array}\right.
 	 \end{equation}
 \end{enumerate}
 \end{theorem}

 This theorem motivates the following definition of  {\it $1$-Laplacian with mixed boundary conditions},  that we denote   by $-\Delta^{\psi,g}_1$.  This definition has sense for $\Vert\psi\Vert_\infty\le 1$.

  \begin{definition} \rm
$v\in -\Delta^{\psi,g}_1 u$ if $u \in  BV(\Omega)\cap L^2(\Omega)$, $v\in L^2(\Omega)$, and there exists $\z \in X_2(\Omega)$ with $\Vert \z \Vert_\infty \leq 1$ satisfying
\begin{equation}\label{E1FundN}
	\left\{ \begin{array}{llll} - \divi \z = v \quad & \hbox{in} \ \ \mathcal{D}'(\Omega), \\[10pt] (\z,Du) = \vert Du \vert \quad &\hbox{as measures},\\[10pt] [\z,\nu_\Omega] = \psi \quad &\hbox{in} \ \ \Gamma_N, \\[10pt] [\z,\nu_\Omega] \in \hbox{sign}(g-u) \quad &\hbox{in} \ \ \Gamma_D.   \end{array}\right.
\end{equation}
\end{definition}

The main future of this operator is the following
\begin{theorem}  \label{ThConpAccre} Under the assumptions of Theorem \ref{ThStat1}, the operator  $-\Delta^{\psi,g}_1$ is a maximal monotone graph in $L^2(\Omega).$  Moreover  $-\Delta^{\psi,g}_1$ coincides with $\partial_{L^2(\Omega} \mathcal{F}_{\psi,g}$, it is completely accretive (see the appendix for the the definition),  and has dense domain.
	\end{theorem}

As consequence of  the above result,  applying the Brezis-Komura Theorem (see Theorem~\ref{Bre-Kom} in the appendix) and having in mind that $-\Delta^{\psi,g}_1$ is completely accretive, we have the following existence and uniqueness result  for problem~\eqref{ProblemI}.

\begin{theorem}\label{inteo310} Let $\psi \in L^\infty(\Gamma_N)$ be such that  $\Vert  \psi \Vert_\infty < 1$ and  $g \in L^1(\Gamma_D)$. For any $u_0 \in L^2(\Omega)$ and any $T > 0$, there exists a unique strong solution of the problem \eqref{ProblemI}, in the sense $u \in C([0,T]; L^2(\Omega))  \cap W^{1,2}_{\rm loc}(0, T; L^2(\Omega))$,   $u(0, \cdot) = u_0$, and, for almost every $t \in (0,T)$,
	\begin{equation}\label{def:dirichletflow}
 u_t(t, \cdot)  - \Delta^{\psi,g}_1u(t,.)\ni 0.
	\end{equation}
That is, for almost every $t \in (0,T)$ there exists a vector field  $\z(t) \in X_2(\Omega)$ with $\| \z(t) \|_\infty \leq 1$ such that  the following conditions hold:
	$$\left\{\begin{array}{l}\displaystyle   u_t(t, .) = \mathrm{div}(\z(t))  \quad \hbox{in} \  \mathcal{D}'(\Omega), \\[12pt] \displaystyle (\z(t), Du(t)) = |Du(t)| \quad \hbox{as measures}, \\[12pt] \displaystyle [\z(t),\nu_\Omega] = \psi  \quad \mathcal{H}^{N-1}\hbox{-a.e.} \ \hbox{on} \ \Gamma_N, \\[12pt] \displaystyle [\z(t),\nu_\Omega] \in \hbox{\rm sign}(g- u(t))  \quad  \mathcal{H}^{N-1}\hbox{-a.e.} \ \hbox{on} \ \Gamma_D.
	\end{array}\right.$$
 	Moreover, the following comparison principle holds: for any $q \in [1,\infty]$, if $u_1, u_2$ are weak solutions for the initial data $u_{1,0}, u_{2,0} \in  L^2(\Omega, \nu) \cap L^q(\Omega, \nu) $ respectively, then
	\begin{equation}\label{DCompPrincipleplaplaceBV}
		\Vert (u_1(t) - u_2(t))^+ \Vert_q \leq \Vert ( u_{1,0}- u_{2,0})^+ \Vert_q.
	\end{equation}
\end{theorem}

Notice that the solution of the evolution problem may be also characterized through variational formulation.  This may be expressed using the  characterization of the operator $ -\Delta^{\psi,g}_1$ in terms of variational inequalities as follows:

  \begin{proposition}\label{characterisationI} Let $\psi \in L^\infty(\Gamma_N)$ be such that  $\Vert  \psi \Vert_\infty \le 1$ and  $g \in L^1(\Gamma_D)$.
 	The following conditions are equivalent:
 	\\
 	(a) $(u,v) \in -\Delta^{\psi,g}_1$;
 	\\
 	(b)  $u \in  BV(\Omega)\cap L^2(\Omega)$, $v\in L^2(\Omega)$, and there exists a vector field $\z \in X_2(\Omega)$ with  $\| \z \|_\infty \leq 1$  such that
 	\begin{equation}\label{divmed}\begin{array}{c}\displaystyle-\mathrm{div}(\z) = v \ \hbox{ in } \  \mathcal{D}'(\Omega),\\[6pt]
 			\displaystyle [\z,\nu_\Omega] = \psi \ \hbox{ in } \Gamma_N,
 	\end{array}\end{equation}
 	and the following variational inequality holds true:
 	\begin{align}\label{eq:variationalequalitydirichletpr}
 		&\int_{\Omega} v(w-u) \, dx \le \int_{\Omega} (\z,Dw) - \int_{\Omega} |Du|
 		\\
 		&\qquad\qquad - \int_{\Gamma_D} [\z,\nu_\Omega] \, (w - g) \, d\mathcal{H}^{N-1} - \int_{\Gamma_D} | u - g| \, d\mathcal{H}^{N-1}-  \int_{\Gamma_N} \psi (w- u)\, d\mathcal{H}^{N-1},
 	\end{align}
 	for every $w \in BV(\Omega) \cap L^2(\Omega)$; \\
 	(c) $u \in  BV(\Omega)\cap L^2(\Omega)$, $v\in L^2(\Omega)$, and there exists a vector field $\z \in X_2(\Omega)$ with  $\| \z \|_\infty \leq 1$  satisfying \eqref{divmed} and the following variational inequality holds true:
 	\begin{align}\label{eq:variationalequalitydirichletprN}
 		&\int_{\Omega} v(w-u) \, dx \le \int_{\Omega} (\z,Dw) - \int_{\Omega} |Du|
 		\\
 		&\qquad\qquad + \int_{\Gamma_D}\vert w - g \vert \, d\mathcal{H}^{N-1} - \int_{\Gamma_D} | u - g| \, d\mathcal{H}^{N-1}-  \int_{\Gamma_N} \psi (w- u)\, d\mathcal{H}^{N-1},
 	\end{align}
 	for every $w \in BV(\Omega) \cap L^2(\Omega)$;
 	\\
 	(d)  $u \in  BV(\Omega)\cap L^2(\Omega)$, $v\in L^2(\Omega)$, and there exists a vector field $\z \in X_2(\Omega)$ with  $\| \z \|_\infty \leq 1$  satisfying \eqref{divmed}
 	and  the following variational equality holds true:
 	\begin{align}\label{eq:variationalequalitydirichlet}
 		&\int_{\Omega} v(w-u) \, dx = \int_{\Omega} (\z,Dw) - \int_{\Omega} |Du|
 		\\
 		&\qquad\qquad - \int_{\Gamma_D} [\z,\nu_\Omega] \, (w - g) \, d\mathcal{H}^{N-1} - \int_{\Gamma_D} | u - g| \, d\mathcal{H}^{N-1}-  \int_{\Gamma_N} \psi ( w-u)\, d\mathcal{H}^{N-1},
 	\end{align}
 	for every $w \in BV(\Omega) \cap L^2(\Omega)$.
 \end{proposition}

 \subsection{The case where $\Vert \psi \Vert_\infty =1$}

   For $k>0$ set $T_k(r)= r$ if $|r|\le k$, $T_k(r)=k\hbox{sign}(r)$ if $|r|>k$.

  \begin{theorem}\label{ThStat2}
  		Let $f\in L^2(\Omega),$  $g\in L^1(\Gamma_D)$ and $\psi\in L^\infty(\Gamma_D)$ satisfying
  		\begin{equation}\label{psicond1}
  			 \Vert \psi\Vert_\infty \le 1.
  	 	\end{equation} The problem \eqref{ProblemIstat}  has a unique solution in the following sense:
  		$u\in  L^2(\Omega)$, $T_k(u)\in BV(\Omega)$ for all $k>0$, and there exists a vector field $\z \in X_2(\Omega)$ with  $\| \z \|_\infty \leq 1$ satisfying
  		\begin{equation}
  			\left\{ \begin{array}{llll} u- \divi \z = f \quad & \hbox{in} \ \ \mathcal{D}'(\Omega), \\[10pt] (\z,DT_k(u)) = \vert DT_k(u) \vert \quad &\hbox{as measures for all $k >0$},\\[10pt] [\z,\nu_\Omega] = \psi \quad &\hbox{in} \ \ \Gamma_N, \\[10pt] [\z,\nu_\Omega] \in \hbox{sign}(T_k(g)-T_k(u)) \quad &\hbox{in} \ \ \Gamma_D, \ \ \hbox{for all $k >0$}.   \end{array}\right.
  		\end{equation}   		
 Moreover,   if $u_1, u_2$ are two solutions corresponding to  $f_1,  f_2\in L^2(\Omega),$ $\psi_1, \psi_2 \in  L^\infty(\Gamma_N)$ and $ g_1,g_2 \in L^1(\Gamma_D),$ respectively, we have:
 \begin{enumerate}
 	\item 	for every $q\in P_{0} ,$
 	 	\begin{equation}\label{CAcret} \int_{\Omega} q(u_1-u_2)(v_1-v_2)\, dx \geq 0; \end{equation}

 	\item if $f_1 \leq f_2$ $\mathcal{L}^{N}$-a.e. in $\Omega$, $g_1 \leq g_2$ $\mathcal{H}^{N-1}$-a.e. in $\Gamma_D$, and $\psi_1 \leq \psi_2$ $\mathcal{H}^{N-1}$-a.e. in $\Gamma_N$, then
 	$$u_1 \leq u_2, \quad \mathcal{L}^N\hbox{-a.e. in }\Omega.$$
 \end{enumerate}
  \end{theorem}

As in the previous section, we can define   the {\it generalized $1$-Laplacian with mixed boundary conditions},  that we denote by $-\widetilde \Delta^{\psi,g}_1,$ as follows.
 \begin{definition}   \rm
$v\in -\widetilde \Delta^{\psi,g}_1 u$ if
  		$u, v\in L^2(\Omega)$, $T_k(u)\in BV(\Omega)$ for all $k>0$, and there exists a vector field $\z \in X_2(\Omega)$ with  $\| \z \|_\infty \leq 1$ satisfying
  		\begin{equation}\label{E1FundNK}
  			\left\{ \begin{array}{llll} - \divi \z = v \quad & \hbox{in} \ \ \mathcal{D}'(\Omega), \\[10pt] (\z,DT_k(u)) = \vert DT_k(u) \vert \quad &\hbox{as measures for all $k >0$},\\[10pt] [\z,\nu_\Omega] = \psi \quad &\hbox{in} \ \ \Gamma_N, \\[10pt] [\z,\nu_\Omega] \in \hbox{sign}(T_k(g)-T_k(u)) \quad &\hbox{in} \ \ \Gamma_D, \ \ \hbox{for all $k >0$}.   \end{array}\right.
  		\end{equation}
  \end{definition}

\begin{remark}\label{enlinfbut} \rm
	
    \noindent 1. It is not difficult to prove that
	$$-  \Delta^{\psi,g}_1 \subset -\widetilde \Delta^{\psi,g}_1,$$
	in the sense that if $v\in -  \Delta^{\psi,g}_1 u$ then $v\in -\widetilde \Delta^{\psi,g}_1 u.$
	Then, in the case $\Vert \psi \Vert_\infty <1$, by maximal monotonicity we have
	$$-  \Delta^{\psi,g}_1 = -\widetilde \Delta^{\psi,g}_1.$$
	
	\noindent 2. If $(u,v)\in -  \widetilde \Delta^{\psi,g}_1 u$ with $u\in BV(\Omega)$ then   $(u,v)\in     - \Delta^{\psi,g}_1$.
	\hfill $\blacksquare$
\end{remark}

The main future of this operator is the following
\begin{theorem}\label{CHARAT}   Under the assumptions of Theorem \ref{ThStat2}, the operator  $-\widetilde \Delta^{\psi,g}_1$ is maximal monotone graph in $L^2(\Omega).$ Moreover,
	\begin{enumerate}
		\item  $-\widetilde \Delta^{\psi,g}_1$
	 it is m-completely accretive.
	
	  \item  $D(-\widetilde \Delta^{\psi,g}_1)$ is dense in $L^2(\Omega)$.
	
	 \item If $u_1, u_2$ satisfy
	 $$   u_i -\widetilde \Delta^1_{\psi_i,g_i}(u_i)\ni f_i, \quad i=1,2,  $$
	 with   $f_1 \leq f_2$ $\mathcal{L}^{N}$-a.e. in $\Omega$, $g_1 \leq g_2$ $\mathcal{H}^{N-1}$-a.e. in $\Gamma_D$, $\psi_1 \leq \psi_2$ $\mathcal{H}^{N-1}$-a.e. in $\Gamma_N$, then
	 $$u_1 \leq u_2,\quad \mathcal{L}^N\hbox{-a.e. in }\Omega.$$
	 \end{enumerate}
\end{theorem}

Again, as a consequence of Theorem \ref{CHARAT}, applying the Brezis-Komura Theorem (Theorem~\ref{Bre-Kom}) and having in mind that $-\widetilde \Delta^{\psi,g}_1$ is completely accretive, we have the following existence and uniqueness result for problem~\eqref{ProblemI}.

\begin{theorem} Let $\psi \in L^\infty(\Gamma_N)$ be such that  $\Vert  \psi \Vert_\infty = 1$ and  $g \in L^1(\Gamma_D)$. For any $u_0 \in L^2(\Omega)$ and any $T > 0$, there exists a unique strong solution of the problem \eqref{ProblemI}, in the sense $u \in C([0,T]; L^2(\Omega))  \cap W^{1,2}_{\rm loc}(0, T; L^2(\Omega))$,   $u(0, \cdot) = u_0$, and, for almost every $t \in (0,T)$,
	\begin{equation}
		u_t(t, \cdot)   -\widetilde \Delta^{\psi,g}_1u(t, .)\ni 0.
	\end{equation}

	Moreover, we  have:
		\begin{equation}   \left\| u(t)\right\|_2 \le \left\|   u_0   \right\|_2+ |\Omega|^{1/2}C_\Omega t ,\end{equation}
and, if $u_0\in L^\infty(\Omega)$, then
 	\begin{equation}
 		   \left\| u(t)\right\|_\infty \le \left\| u_0\right\|_\infty+C_\Omega t,\end{equation}
		being  $C_\Omega$ the constant in Theorem~\ref{pepesab}.

\end{theorem}

 \subsection{Large solution vs  $\frac{Du}{\vert Du \vert} \cdot \nu= 1$ on the boundary }

 Our main results here concerns the large solution of the $1-$Laplacian with mixed boundary conditions. Before to treat the general case, let us begin with  the simple situation where  $\Gamma_D=\partial \Omega,$ i.e. $\Gamma_N=\emptyset.$
 Observe that, formally,   if $g=+\infty$ on $\partial\Omega$,   the generalized Dirichlet boundary condition $$[\z,\nu_\Omega] \in \hbox{sign}(g-u) \ \hbox{ in} \   \partial\Omega,$$
 may be connected to the Neuman boundary condition
 $$[\z,\nu_\Omega]=1\ \hbox{ in} \   \partial\Omega.$$
 This leads, in fact, to the  extremal  case (with $\Gamma_N=\partial\Omega$)
 \begin{equation}\label{LargeSolMP}
 	\left\{
 	\begin{array}{ll}
 		u_t =  \Delta_1 u   &\hbox{in} \ \ (0,T) \times\Omega,
 		\\[10pt] \frac{Du}{\vert Du \vert} \cdot \nu= 1 \quad &\hbox{on} \ \ (0,T) \times \partial \Omega,
 		\\[10pt] u(0) = u_0.
 	\end{array}
 	\right.
 \end{equation}
 This last problem was indeed considered in~\cite{MP} to be understood as the solution to the Dirichlet problem with $g=+\infty$:
 \begin{equation}\label{LargeSol22}
 	\left\{
 	\begin{array}{ll}
 		u_t =  \Delta_1 u   &\hbox{in} \ \ (0,T) \times\Omega,
 		\\[10pt] u = +\infty   &\hbox{on} \ \ (0,T) \times\partial\Omega,
 		\\[10pt] u(0) = u_0,
 	\end{array}
 	\right.
 \end{equation}
 whose solution is called   {\it large solution} for the $1$-Laplacian flow. Moreover, this solution  was obtained by taking limits as $n\to+\infty$ on an approximated Dirichlet problem with $g=n$ on the boundary.

 Using the comparsion result stated at Theorem~\ref{CHARAT}(3) for the elliptic problem, we can pass to the evolution problem via Theorem~\ref{TeoConBP} to see that these {\it large solutions} are definitely  the {\it largest solutions}. Indeed, taking   $u_\psi$ to be the strong solution of
 \begin{equation}\label{LargeSolMPparacomp}
 	\left\{
 	\begin{array}{ll}
 		u_t =  \Delta_1 u   &\hbox{in} \ \ (0,T) \times\Omega,
 		\\[10pt] \frac{Du}{\vert Du \vert} \cdot \nu= \psi \quad &\hbox{on} \ \ (0,T) \times \partial \Omega,
 		\\[10pt] u(0) = u_0,
 	\end{array}
 	\right.
 \end{equation}
 then $$u_\psi\le u_1,$$ where $u_1$ is the {\it large solution.}

 This idea may be generalized to the case with  mixed  boundary condition. To this aim let us assume now that
 \begin{equation}
 \Gamma_D = \Gamma_{D,1}\cup\Gamma_{D,2}\cup\Gamma_{D,3}
 \end{equation} where $\Gamma_{D_i}$ are mutually disjoint, and   consider the problem
 \begin{equation}\label{LargeSol2233}
 	\left\{
 	\begin{array}{ll}
 		u_t =  \Delta_1 u   &\hbox{in} \ \ (0,T) \times\Omega,
 		\\[10pt] \frac{Du}{\vert Du \vert} \cdot\nu = \psi \quad &\hbox{on} \ \ (0,T) \times\Gamma_N,
 		\\[10pt] u = g   &\hbox{on} \ \ (0,T) \times\Gamma_{D,1},
 		\\[10pt] u = +\infty   &\hbox{on} \ \ (0,T) \times\Gamma_{D,2},
 		\\[10pt] u = -\infty   &\hbox{on} \ \ (0,T) \times\Gamma_{D,3},
 		
 		\\[10pt] u(0) = u_0.
 	\end{array}
 	\right.
 \end{equation}
 Our definition of the largest solution, i.e., the solution of the problem \eqref{LargeSol2233},    is closely  connected to the solution, for large $n\in \mathbb{N},$ of the following mixed problem
 \begin{equation}\label{LargeSol2233enun}
 	\left\{
 	\begin{array}{ll}
 		u_t =  \Delta_1 u   &\hbox{in} \ \ (0,T) \times\Omega,
 		\\[10pt] \frac{Du}{\vert Du \vert} \cdot\nu = \psi \quad &\hbox{on} \ \ (0,T) \times\Gamma_N,
 		\\[10pt] u = g   &\hbox{on} \ \ (0,T) \times\Gamma_{D,1},
 		\\[10pt] u = n   &\hbox{on} \ \ (0,T) \times\Gamma_{D,2},
 		\\[10pt] u = -n  &\hbox{on} \ \ (0,T) \times\Gamma_{D,3},
 		\\[10pt] u(0) = u_0,
 	\end{array}
 	\right.
 \end{equation}
More precisely, we have:

\begin{theorem}\label{teolsol01}  	Let $\psi \in L^\infty(\Gamma_N)$ be such that  $\Vert  \psi \Vert_\infty \leq 1$,   $g \in L^1(\Gamma_{D_1})$ and $u_0\in L^2(\Omega)$.  The problem \eqref{LargeSol2233} has a unique solution $u_L$ in the sense that
	 $$u_L(t)=\lim_{n\to\infty}u_{n}(t) \ \ \hbox{uniformly on} \ [0,T],$$
	where, for each $n\in \mathbb{N}$,    $u_{n}$ is  the strong solution of \eqref{LargeSol2233enun}.
	Moreover,   $u_L$  is also the strong solution of 	the problem
	\begin{equation}\label{LargeSol2233ch}
		\left\{
		\begin{array}{ll}
			u_t =  \Delta_1 u   &\hbox{in} \ \ (0,T) \times\Omega,
			\\[10pt] \frac{Du}{\vert Du \vert} \cdot\nu = \widetilde{\psi} \quad &\hbox{on} \ \ (0,T) \times\widetilde{\Gamma}_N,
			\\[10pt] u = g   &\hbox{on} \ \ (0,T) \times\Gamma_{D,1},
			\\[10pt] u(0) = u_0,
		\end{array}
		\right.
	\end{equation} with
	$$\widetilde{\Gamma}_N=\Gamma_N\cup\Gamma_{D,2}\cup\Gamma_{D,3}\quad \hbox{ and }\quad   \widetilde{\psi}=\psi\1_{\Gamma_N}+1\1_{\Gamma_{D,2}}-1\1_{\Gamma_{D,2}}.$$

\end{theorem}

\begin{corollary}
 	Let $\psi \in L^\infty(\Gamma_N)$ be such that  $\Vert  \psi \Vert_\infty \leq 1$,  $g \in L^1(\Gamma_{D}),$ and $u_0\in L^2(\Omega)$.  If $u$ is a weak solution of  the problem \eqref{ProblemI}, then
		$$\underline u \leq u\leq \overline u\quad \hbox{ a.e. in }Q,$$
		where $\overline u$ and  $\underline u$  are  the solutions of
		\begin{equation}
			\left\{
			\begin{array}{ll}
				\overline 	u_t =  \Delta_1 \overline u   &\hbox{in} \ \ (0,T) \times\Omega,
				\\[10pt] \frac{D\overline u}{\vert D\overline u \vert} \cdot\nu = \overline  {\psi} \quad &\hbox{on} \ \ (0,T) \times {\partial \Omega},
				\\[10pt] \overline u(0) = u_0,
			\end{array}
			\right.
			\quad 	\hbox{and }
			\left\{
			\begin{array}{ll}
				\underline u_t =  \Delta_1 \underline u   &\hbox{in} \ \ (0,T) \times\Omega,
				\\[10pt] \frac{D\underline u}{\vert D\underline u \vert} \cdot\nu = \underline {\psi} \quad &\hbox{on} \ \ (0,T) \times {\partial \Omega},
				\\[10pt]\underline  u(0) = u_0,
			\end{array}
			\right.
		\end{equation}
respectively, where
$$\overline \psi= \psi\chi_{\Gamma_N} +  \chi_{\Gamma_D}    \quad \hbox{ and }  \underline  \psi= \psi\chi_{\Gamma_N} -  \chi_{\Gamma_D}   .$$
\end{corollary}
\begin{proof} Using the comparison  principle of Theorem \ref{ThStat2}, for any $n\geq \Vert g\Vert_\infty,$ we see that $u$ satisfies
	$$u_{1n}\leq u\leq u_{2n},\quad \hbox{ a.e. in }Q,$$
	where $u_1$ and $u_2$ are the solutions of the problems
	\begin{equation}
		\left\{
		\begin{array}{ll}
			u_{{1n}t} =  \Delta_1 u_{1n}   &\hbox{in} \ \ (0,T) \times\Omega,
			\\[10pt] \frac{Du_{1n}}{\vert Du_{1n} \vert} \cdot\nu = \psi \quad &\hbox{on} \ \ (0,T) \times\Gamma_N,
			\\[10pt] u_{1n} = -n &\hbox{on} \ \ (0,T) \times\Gamma_D,
			\\[10pt] u_{1n}(0) = u_0,
		\end{array}
		\right.
  \hbox{ and }
  	\left\{
		\begin{array}{ll}
			u_{{2n}t} =  \Delta_1 u_{2n}  &\hbox{in} \ \ (0,T) \times\Omega,
			\\[10pt] \frac{Du_{2n}}{\vert Du_{2n} \vert} \cdot\nu = \psi \quad &\hbox{on} \ \ (0,T) \times\Gamma_N,
			\\[10pt] u_{2n} = n &\hbox{on} \ \ (0,T) \times\Gamma_D,
			\\[10pt] u_{2n}(0) = u_0,
		\end{array}
		\right.
	\end{equation} respectively.  Then the result of the corollary follows by letting $n\to\infty $ and using Theorem \ref{teolsol01}.
\end{proof}

\section{Proofs for the case where $ \Vert \psi\Vert_\infty <1. $ }\label{sectionla4}

\subsection{Duality and its consequence: characterization of $\Delta^{\psi,g}_1$}

Using Fenchel-Rocka\-fellar duality, we define and characterize the main operator for evolution the problem~\eqref{ProblemI}. To ensure completeness,   we briefly present some of the convex duality methods  related to calculus of variations, in particular the Fenchel-Rockafellar duality theorem. Our presentation follows the one in \cite{ET} (in particular Chapters III and V).

 	Given a Banach space $V$ and a convex function $F: V \rightarrow \mathbb{R} \cup \{ + \infty \}$, we define its {\it Legendre-Fenchel transform} (or conjugate function) $F^*: V^* \rightarrow \mathbb{R} \cup \{ + \infty \}$ by the formula
		\begin{equation*}
			F^*(v^*) = \sup_{v \in V} \bigg\{ \langle v, v^* \rangle - F(v) \bigg\}.
		\end{equation*}

We now state the Fenchel-Rockafellar duality theorem in the form suitable for calculus of variations and presented in \cite{ET}. Let $X,Y$ be two Banach spaces and let $A: X \rightarrow Y$ be a continuous linear operator. Denote by $A^*: Y^* \rightarrow X^*$ to the dual operator of $A$. Then, for the primal minimisation problem
\begin{equation}\tag{\bf P}\label{Ap.B.eq:primal}
	\hbox{minimize}\bigg\{ E(Au) + G(u) : u \in X \bigg\},
\end{equation}
its dual problem is defined as the maximisation problem
\begin{equation}\tag{\bf P*}\label{Ap.B.eq:dual}
	\hbox{maximize} \bigg\{ - E^*(-p^*) - G^*(A^* p^*) : p^* \in Y^* \bigg\},
\end{equation}
where $E^*$ and $G^*$ are the Legendre–Fenchel transformations (conjugate  functions)  of $E$ and $G$ respectively. The following result holds.

\begin{theorem}[Fenchel-Rockafellar duality theorem]\label{FRTh}
	Assume that $E$ and $G$ are proper, convex and lower semicontinuous. If there exists $u_0 \in X$ such that $E(A u_0) < \infty$, $G(u_0) < \infty$ and $E$ is continuous at $A u_0$, then
	$$\eqref{Ap.B.eq:primal} = \eqref{Ap.B.eq:dual},$$
	and the dual problem \eqref{Ap.B.eq:dual} admits at least one solution. Moreover, the following optimality conditions between these two problems is satisfied:
	\begin{equation}\label{Ap.B.optimality}
		A^* p^* \in \partial G(u) \quad \mbox{and} \quad -p^* \in \partial E(Au)
	\end{equation}
	when $u$ is a solution of \eqref{Ap.B.eq:primal} and $p^*$ is a solution of \eqref{Ap.B.eq:dual}, or, equivalently,
	\begin{equation}\label{Ap.B.eq.extremality1}
		E(Au) + E^*(-p^*) = \langle -p^*, Au \rangle
	\end{equation}
	and
	\begin{equation}\label{Ap.B.eq.extremality2}
		G(u) + G^*(A^* p^*) = \langle u, A^* p^* \rangle.
	\end{equation}
\end{theorem}

Now,   to prove the main result Theorem~\ref{ThStat1} we use the  next three lemmas.
\begin{lemma}\label{elprimlem}
 Under the assumptions of Theorem~\ref{ThStat1}, the optimization problem $(P)$  has a unique solution $ u\in BV(\Omega)\cap L^2(\Omega)$.
		\end{lemma}
	\begin{proof}
		We have that $\Phi_f$ is strictly convex and lower semi-continuous.  Let us see that it is coercive. Indeed, by Theorem~\ref{pepesab},
		$$ \begin{array}{c} \displaystyle
			\Phi_f (v) \geq \int_\Omega \vert Dv \vert - \Vert \psi\Vert_\infty \Vert v \Vert_{L^1(\partial\Omega)} - \Vert f \Vert_2 \Vert v \Vert_2 + \frac12 \Vert v \Vert_2^2 \\[12pt]
			\displaystyle
			\geq (1 - \Vert \psi\Vert_\infty)\int_\Omega \Vert Dv \Vert - C_\Omega \Vert \psi\Vert_\infty \Vert v \Vert_{L^1(\Omega)} -\Vert f \Vert_2 \Vert v \Vert_2 + \frac12 \Vert v \Vert_2^2  \\[12pt]
			\displaystyle  \geq -Q \Vert v \Vert_2  + \frac12 \Vert v\Vert_2^2.
		\end{array}$$
		Now, the above three conditions  give us  the existence of a unique minimizer $  u \in BV(\Omega)\cap L^2(\Omega)$ of the functional $\Phi_f$.
			\end{proof}
		
  Next, we introduce a weak version of (P) to work with a regular variable in $W^{1,1}(\Omega)\cap L^2(\Omega).$  This approach first lets us tackle the problem using Fenchel-Rockafellar duality (as detailed in Lemma \ref{Lduality1}), and then, via direct computation, we prove Theorem \ref{ThStat1} (as shown in Lemma \ref{lemmaproof}).

		\begin{lemma}\label{Lduality1}
			Under the assumptions of Theorem \ref{ThStat1},
			 $$ (M)= (\tilde P),$$
			where
				\begin{equation}\label{ResOptimvez2}
				(\tilde P):=  \inf_{u\in W^{1,1}(\Omega)\cap L^2(\Omega) }\left\{  \int_\Omega \vert \nabla  v \vert -\int_{\Gamma_N} \psi v  \, d  \mathcal{H}^{N-1} + \int_{\Gamma_D} \vert g - v \vert  \, d  \mathcal{H}^{N-1}  + \frac{1}{2} \int_{ \Omega}(v- f)^2 \, dx  \right\},
			\end{equation}
\end{lemma}
\begin{proof} Set $$U =  W^{1,1}(\Omega)\cap L^2(\Omega).$$
We have that $U$ is a Banach space respect to the norm
			$$\Vert u \Vert_U:= \max \{ \| u \|_{W^{1,1}(\Omega)}, \| u_2 \|_{L^2(\Omega)} \}.$$
			And, since $C_c^\infty(\Omega) \subset W^{1,1}(\Omega) \cap L^2(\Omega)$, and it is a dense subset of both $W^{1,1}(\Omega)$ and $L^2(\Omega)$, by \cite[Theorem 2.7.1]{BL}, we have \begin{equation}
				U^* = ( W^{1,1}(\Omega))^* + L^2(\Omega),
			\end{equation}
			whose norm is given by
			$$ \Vert u^* \Vert_{U^*} = \inf \{\| u_1^* \|_{(W^{1,1}(\Omega))^*}+ \| u^*_2 \|_{L^2(\Omega)}   :   u^*  = u_1^* + u_2^* \}.$$
 Set
 $$  V = L^1(\partial\Omega,d \mathcal{H}^{N-1} )\times L^2(\Omega) \times L^1(\Omega,\R^N).$$
We denote the points $p \in V$ in the following way, $p = (p_0,p_1, \overline{p})$, where $p_0 \in L^1(\partial\Omega,d \mathcal{H}^{N-1} ),$ $p_1\in L^2(\Omega)$ and $\overline{p} \in L^1(\Omega;\mathbb{R}^N)$. We also use a similar notation for points $p^* \in V^*$.

Let $E: V \rightarrow \mathbb{R}$ be  given by the formula
\begin{equation}\label{Ieq:definitionofEND}
			E(p_0, \overline{p}) = E_0(p_0) + E_1(p_1)+ E_2(\overline{p}),
\end{equation}
			with
			$$E_0(p_0) =  -\int_{\Gamma_N} \psi p_0 \,d\mathcal{H}^{N-1} +  \int_{\Gamma_D} \vert g- p_0 \vert  \, d  \mathcal{H}^{N-1}, $$
			$$E_1(p_1)= \frac{1}{2} \int_\Omega p_1^2 \: dx $$   and  $$E_2(\overline{p}) := \int_\Omega \Vert \overline{p} \Vert \, dx,$$
			where $||.||$ is the Euclidean norm in $\mathbb{R}^N$.
			Set also $G:U\rightarrow \mathbb{R}$ given by
			$$ G(u)=-\int_\Omega f\: u +  \frac12 \int_\Omega  f^2 \, dx .$$
			And define  the operator $A: U \rightarrow V$  by the formula
			$$ Au = ( Tr (u),-u,-  \nabla u),$$
			which is linear and continuous.  Clearly, we have
			\begin{equation}  \label{eq:primal}
				(\tilde P)= \inf_{u \in U} \bigg\{ E(Au) + G(u) \bigg\}.
			\end{equation}
			Moreover, its dual problem is  the maximisation problem
			\begin{equation} \label{dual1N}
				(\tilde P^*)= 	\sup_{p^* \in L^\infty(\partial\Omega, \mathcal{H}^{N-1})\times L^2(\Omega) \times L^\infty(\Omega; \mathbb{R}^N)} \bigg\{  - E_0^*(-p_0^*) -E_1^*(-p_1^*)- E_2^*(-\overline{p}^*) - G^*(A^* p^*) \bigg\},
			\end{equation}
			where $E_i^*$, $i=0,1,2$, and $G^*$ are the Legendre–Fenchel transformations   of $E_i$, $i=0,1,2$, and $G$ respectively.
			  Since  for  $u_0 =0$ we have $E(Au_0) = 0 < \infty$, $G(u_0) = 0 < \infty$ and $E$ is continuous at $Au_0$, by the Fenchel-Rockafellar Duality Theorem, we have
			\begin{equation}\label{DFR1-TVflowN} (\tilde P)= (\tilde P^*)
			\end{equation}
			and 	the dual problem $(\tilde P^*)$ admits at least one solution, that is a maximizer. Let us prove that actually
			$$ (\tilde P^*) =(M),$$
              which gives our statement
             $(\tilde P)=(M).$

  See that the functional $E_0^*: L^\infty(\partial\Omega, \mathcal{H}^{N-1}) \rightarrow \mathbb{R} \cup \{ \infty \}$ is given by the formula
			$$E_0^*(-p_0^*) = \left\{ \begin{array}{ll} \displaystyle -\int_{\Gamma_D}gp_0^*d \mathcal{H}^{N-1}   &\hbox{if}\
				\left\{\begin{array}{l}p_0^* =   \psi \ \hbox{$\mathcal{H}^{N-1}$-a.e. on $\Gamma_N$} \  \hbox{and}\\[6pt]  \vert   p_0^* \vert \leq 1 \ \ \hbox{$\mathcal{H}^{N-1}$-a.e. on $\Gamma_D$,}\end{array}\right.  \\[16pt] +\infty,   & \hbox{otherwise}.   \end{array}  \right. $$
				Indeed,
				\begin{eqnarray*}
			E_0^*(-p_0^*) &=& \max_{p_0\in L^1(\partial\Omega, \mathcal{H}^{N-1})  } \left\{ - \int_{\partial \Omega } p_0 \: p_0^*
		+\int_{\Gamma_N} \psi p_0 \,d\mathcal{H}^{N-1} - \int_{\Gamma_D} \vert g- p_0 \vert  \, d  \mathcal{H}^{N-1} \right\} \\  \\
		&=& \max_{p_0\in L^1(\partial\Omega, \mathcal{H}^{N-1})  } \left\{ - \int_{\Gamma_N } p_0 \: (p_0^*-\psi)
	  - \int_{\Gamma_D} ( p_0^*\: p_0 +\vert g- p_0 \vert)\, d  \mathcal{H}^{N-1} \right\}  \\  \\
	  &=& \max_{p_0\in L^1(\partial\Omega, \mathcal{H}^{N-1})  } \left\{ - \int_{\Gamma_N } p_0 \: (p_0^*-\psi)
	  - \int_{\Gamma_D} ( p_0^*(p_0-g) +\vert g- p_0 \vert)\, d  \mathcal{H}^{N-1} \right\}  \\  \\
	  &  &  -   \int_{\Gamma_D}   p_0^* g 	\\  \\
	   &=& \max_{p_0\in L^1(\partial\Omega, \mathcal{H}^{N-1})  } \left\{ - \int_{\Gamma_N } p_0 \: (p_0^*-\psi)
	  +\int_{\Gamma_D} \vert g-p_0\vert (  p_0^*\hbox{sign}_0  (g- p_0) -1)\, d  \mathcal{H}^{N-1} \right\}  \\  \\
	  &  &  -   \int_{\Gamma_D}   p_0^* g 	\\  \\
	  &=& \left\{ \begin{array}{ll} \displaystyle -\int_{\Gamma_D}gp_0^*d \mathcal{H}^{N-1}   &\hbox{if}\
	  	\left\{\begin{array}{l}p_0^* =   \psi \ \hbox{$\mathcal{H}^{N-1}$-a.e. on $\Gamma_N$} \  \hbox{and}\\[6pt]  \vert   p_0^* \vert \leq 1 \ \ \hbox{$\mathcal{H}^{N-1}$-a.e. on $\Gamma_D$,}\end{array}\right.  \\[16pt] +\infty,   & \hbox{otherwise}.   \end{array}  \right. \end{eqnarray*}
	 	The functional $E_1^*:  L^2(\Omega) \rightarrow \mathbb{R} \cup \{ \infty \}$ is given by the formula		
							$$E_1^*(p_1^*)= \frac{1}{2} \int_\Omega p_1^{*2} \: dx $$
			And the functional $E_2^*: L^\infty(\Omega; \mathbb{R}^N) \rightarrow [0,\infty]$ is given by
			$$E_2^*(\overline{p}^*) = \left\{ \begin{array}{ll} 0  &\hbox{if} \ \ \Vert \overline{p}^* \Vert_{L^\infty(\Omega, \R^N)} \leq 1, \\[10pt] +\infty,   & \hbox{otherwise}.  \end{array}  \right.
			$$
			This implies that  $(\tilde P^*)$ is equal  to
			            \begin{equation}
				\begin{array}{c}\displaystyle	\max_{p^* \in L^\infty(\partial\Omega, \mathcal{H}^{N-1})  \times L^2(\Omega) \times L^\infty(\Omega; \mathbb{R}^N)} \bigg\{ \int_{\Gamma_D}gp_0^*d \mathcal{H}^{N-1} - \frac{1}{2} \int_\Omega p_1^{*2} \: dx  - G^*(A^* p^*)  :  \quad   \\    \hfill  p_0^* =   \psi , \hbox{ on }\Gamma_N,\:  \vert   p_0^* \vert \leq 1,\hbox{ on }\Gamma_D,\: \Vert \overline{p}^* \Vert_{L^\infty(\Omega, \R^N)} \leq 1   \bigg\}.\end{array}
			\end{equation}
			Now, we see that, for any $p^* \in V^*,$
			  \begin{eqnarray*}
				G^*(A^* p^*) &=& \max_{ p\in U} \left( \langle A^*p^*,p\rangle_{ U^*,U} -G(p)\right)\\  \\
				&=& \max_{ p\in U}  \left(\langle  p^*,Ap\rangle_{V^*,V} -G(p)\right)\\  \\
&=& 	 \max_{ p\in U}    \left(	 \int_{\partial \Omega} p_0^* \: Tr (p) -\int_\Omega p_1^*\: p  -\int_\Omega \overline p^*\cdot \nabla p \, dx 	+\int_\Omega p\: f\: dx \right) - \frac{1}{2} \int_\Omega f^2\: dx .
  	\end{eqnarray*}
			This is finite and is equal to {$ \displaystyle -\frac{1}{2}\int_\Omega f^2\: dx$} if and only if the triplet $(p_0^*,p_1^*,\overline p^*)\in  L^1(\partial\Omega,d \mathcal{H}^{N-1} )\times L^2(\Omega) \times L^1(\Omega,\R^N)$ is  such that
			$$	 \int_{\partial \Omega} p_0^* \: Tr (p) -\int_\Omega p_1^*\: p  -\int_\Omega \overline p^*\cdot \nabla p  \, dx 	+\int_\Omega p\: f\: dx =0 \quad \hbox{for any }p\in U.$$
			That is $(p_0^*,p_1^*,\overline p^*)\in  L^1(\partial\Omega,d \mathcal{H}^{N-1} )\times L^2(\Omega) \times L^1(\Omega,\R^N)$  satisfies the PDE problem
	 \begin{equation}\label{PDEdual}
	 		 \left\{  \begin{array}{ll}
			p_1^*-\divi \overline p^* =f  \quad & \hbox{ in }\Omega\\[8pt]
		\overline{p}^*\cdot \nu_\Omega = p_0^* & \hbox{ on }\Gamma_N.
			\end{array}
		\right.
	 \end{equation}

 Thus  $(\tilde P^*)$  is equal to
 	\begin{equation}
 	\begin{array}{c}   \displaystyle	\max_{p^* \in L^\infty(\partial\Omega, \mathcal{H}^{N-1})  \times L^2(\Omega) \times L^\infty(\Omega; \mathbb{R}^N)} \bigg\{     \int_{\Gamma_D}gp_0^*d \mathcal{H}^{N-1}   - \frac{1}{2} \int_\Omega p_1^{*2} \: dx +  \frac{1}{2}\int_\Omega f^2\: dx   :   \quad    \\   \hfill p_0^* =  \psi , \hbox{ on }\Gamma_N,\:   \vert   p_0^* \vert \leq 1,\hbox{ on }\Gamma_D,\: \Vert \overline{p}^* \Vert_{L^\infty(\Omega, \R^N)} \leq 1,\: (p_1^*,\overline p^*) \hbox{ satisfies }\eqref{PDEdual}   \bigg\},
    \end{array}
 \end{equation}
 that is,   $$(\tilde P^*)=(M),$$
 and the proof is finished.
\end{proof}

\begin{lemma} \label{lemmaproof}
Under the assumptions of Theorem \ref{ThStat1}, we have
  \begin{equation}\label{equality0}
 	(P)=(M).
\end{equation}

Moreover, if  $u$ is a solution of $(M)$ and $(\z,\xi)$ is a solution of $(P^*)$, then $\xi=u$ and the couple $(u,\z)$ solves the PDE problem~\eqref{ProblemIstat}.
\end{lemma}
\begin{proof}
First, combining   Lemma~\ref{Lduality1} with the fact that  $U\subset BV(\Omega)\cap L^2(\Omega),$ we have
 \begin{equation} \label{ineq0}
	 (P)\leq  (\tilde P)=(M).
\end{equation}
 On the other one sees that, for any $v \in BV(\Omega),$ and $z\in L^\infty(\Omega)^N$ such that $\Vert z\Vert_\infty  \leq 1,\  -\divi   z =f- \xi $ in $\Omega,$ and $  [z,\nu_\Omega] = \psi $ in   $\Gamma_N,$  by Green's formula we have
 $$\int_\Omega (z,Dv) +\int_\Omega \xi\: v \, dx  = \int_\Omega f\: v \, dx  +\int_{\Gamma_N}  \psi v  \, d  \mathcal{H}^{N-1}  + \int_{\Gamma_D}  [z,\nu_\Omega]  \: v \, d  \mathcal{H}^{N-1} .$$
Then,
    $$\int_\Omega \vert Dv\vert   +\frac{1}{2}\int_\Omega  (v-f)^2 \, dx   + \int_{\Gamma_D}  \vert g-v \vert   \, d  \mathcal{H}^{N-1}   -\int_{\Gamma_N}     \psi v  \, d  \mathcal{H}^{N-1}    $$   $$ + \int_\Omega \underbrace{   ((z,Dv) -\vert Dv\vert   ) }_{=:I_1(z,Dv) }+ \int_\Omega \underbrace{   (  \xi\: v -  \frac{1}{2} v^2  )}_{=:I_2(v,\xi) } \, dx    + \int_{\Gamma_D}\underbrace{   \left(  [z,\nu_\Omega]  \: (g-v) -  \vert g-v \vert   \right) }_{=:I_3(z,v)  } \, d  \mathcal{H}^{N-1}  $$
   $$ =	 \frac{1}{2} \int_{ \Omega} f^2  \, dx -  \frac{1}{2} \int_{ \Omega} \xi^2  \, dx + \int_{\Gamma_D}  [z,\nu_\Omega]  \: g \, d  \mathcal{H}^{N-1}$$
    Using Proposition~\ref{Anzelotti}, Young inequality and the fact that $\vert  [z,\nu_\Omega]  \vert \leq 1 ,$ $\mathcal{H}^{N-1}$-a.e. on $\partial\Omega$,  we have  $I_1(z,Dv) \leq 0$,  $I_2(v,\xi)\leq 0 $ and  $I_3(z,v)\leq 0,$   and then
   $$\int_\Omega \vert Dv\vert   +\frac{1}{2}\int_\Omega  (v-f)^2 \, dx   + \int_{\Gamma_D}  \vert g-v \vert   \, d  \mathcal{H}^{N-1}   -\int_{\Gamma_N}     \psi v  \, d  \mathcal{H}^{N-1}$$    $$\geq \int_\Omega \vert Dv\vert   +\frac{1}{2}\int_\Omega  (v-f)^2 \, dx   + \int_{\Gamma_D}  \vert g-v \vert   \, d  \mathcal{H}^{N-1}   -\int_{\Gamma_N}     \psi v  \, d  \mathcal{H}^{N-1}    $$   $$ + \int_\Omega  I_1(z,Dv)  + \int_\Omega  I_2(v,\xi)  \, dx    + \int_{\Gamma_D} I_3(z,v)  \, d  \mathcal{H}^{N-1}  $$
    $$ =  \frac{1}{2} \int_{ \Omega} f^2  \, dx 	-  \frac{1}{2} \int_{ \Omega} \xi^2  \, dx + \int_{\Gamma_D}  [z,\nu_\Omega]  \: g \, d  \mathcal{H}^{N-1}.$$
Then minimizing in $v$, we get
 $$(P) \geq  \frac{1}{2} \int_{ \Omega} f^2  \, dx 	-  \frac{1}{2} \int_{ \Omega} \xi^2  \, dx + \int_{\Gamma_D}  [z,\nu_\Omega]  \: g \, d  \mathcal{H}^{N-1}.$$
Now, maximizing in $(\xi,z)$, we obtain that
\begin{equation}\label{alas1028}
(M)\leq    (P).
\end{equation}
   Hence,  by  inequalities~\eqref{ineq0} and~\eqref{alas1028},
$$(\tilde P)=(P)=(M).$$
Finally, taking  $u$   a solution of $(P)$ and $(\z,\xi)$   a solution of $(M)$, one sees that $I_1(\z,Du) = 0$,  $I_2(u,\xi)= 0 $ and  $I_3(\z,u)= 0,$  which implies that $(u,\z)$ is a solution of   the PDE problem~\eqref{ProblemIstat}.  \hfill
 \end{proof}

 \begin{remark}\rm
It is not clear how the result in  Theorem \ref{ThStat1} can be achieved
 using standard Fenchel-Rockafellar duality (Theorem \ref{FRTh}), primarily because the dual of the $BV$ space lacks a rigorous characterization. To circumvent this challenge, we introduce the intermediate problem $(\tilde P),$  as demonstrated in the proof of Lemma \ref{lemmaproof}. This approach allows us to relax the problem and effectively leverage Fenchel-Rockafellar duality. Subsequently, we can reconnect with the original problem $(P)$ in the $BV$ space. Essentially, introducing  $(\tilde P)$ helps us avoid the complexities of operating within the inaccessible dual of the $BV$ space.
 \hfill $\blacksquare$	
 	\end{remark}

 	\begin{proof}[Proof of Theorem \ref{ThStat1}] The proof of existence and the characterization of the solutions of the problems $(P)$ and $(M)$ in terms of a solution of the PDE \eqref{StatPDE} follows by   Lemma~\ref{elprimlem}, Lemma \ref{Lduality1} and Lemma \ref{lemmaproof}. The uniqueness  follows   by the strict convexity of the functional $\Phi_f$ as stated in   Lemma~\ref{elprimlem}.
 	\end{proof}

\subsection{Nonlinear semigroup techniques for   existence of   solution of the evolution problem}

\begin{lemma}\label{contenido}  We have $ {-\Delta }^{\psi,g}_1 \subset \partial_{L^2(\Omega)} \mathcal{F}_{\psi,g}$.
	\end{lemma}
	\begin{proof}   Let   $(u,v) \in  {-\Delta }^{\psi,g}_1$,  then given $ w \in BV(\Omega)\cap L^2(\Omega)$, multiplying the first equation in \eqref{E1FundN} by $w-u$ and applying Green's formula, we get, taking into account that $[\z, \nu_\Omega] (u -g)= -\vert g- u \vert$ in $\Gamma_D$,
$$\int_\Omega v(w-u) dx = - \int_\Omega \divi \z (w-u) dx $$ $$= \int_\Omega (\z, Dw) -\int_{\partial \Omega} [\z, \nu_\Omega] w \, d\mathcal{H}^{N-1}- \int_\Omega \vert Du \vert + \int_{\partial \Omega} [\z, \nu_\Omega] u \, d\mathcal{H}^{N-1}$$
$$=\int_\Omega (\z, Dw) -\int_{\Gamma_N} \psi w  \, d  \mathcal{H}^{N-1} - \int_{\Gamma_D} [\z, \nu_\Omega] (w -g)  \, d  \mathcal{H}^{N-1} $$ $$- \int_\Omega \vert Du \vert + \int_{\Gamma_N} \psi u  \, d  \mathcal{H}^{N-1}  +  \int_{\Gamma_D} [\z, \nu_\Omega] (u -g)  \, d  \mathcal{H}^{N-1} $$ $$\leq \int_\Omega \vert Dw \vert -\int_{\Gamma_N} \psi w  \, d  \mathcal{H}^{N-1} + \int_{\Gamma_D} \vert g- w \vert  \, d  \mathcal{H}^{N-1}$$
$$- \int_\Omega \vert Du \vert + \int_{\Gamma_N} \psi u  \, d  \mathcal{H}^{N-1} - \int_{\Gamma_D} \vert g- u \vert  \, d  \mathcal{H}^{N-1} $$
$$\leq \mathcal{F}_{\psi,g}(w) - \mathcal{F}_{\psi,g}(u).$$
Therefore, $ {-\Delta }^{\psi,g}_1 \subset \partial_{L^2(\Omega)} \mathcal{F}_{\psi,g}$.
 \end{proof}

  \begin{lemma}\label{CAoperator}
  	$ {-\Delta }^{\psi,g}_1$ is completely accretive.
  \end{lemma}
  \begin{proof}  By Proposition \ref{prop:completely-accretive}, to  prove that the operator $ {-\Delta }^{\psi,g}_1$ is completely accretive,  we need to show that
  	$$ \int_{\Omega} T(u_1-u_2)(v_1-v_2)\, dx \geq 0 $$
  	for every $T\in P_{0}$ and every $(u_i,v_i) \in  {-\Delta }_{\psi,g}^1$,   $i=1,2$ .
  	
  	Since $(u_i, v_i) \in  {-\Delta }_{\psi,g}^1$,   $i =1,2$, then, $u_i \in  BV(\Omega)$ and there exists $\z_i \in X_2(\Omega)$ with $\Vert \z_i \Vert_\infty \leq 1$ satisfying:
  	\begin{equation}\label{E1FundNCA}
  		\left\{ \begin{array}{llll} - \divi \z_i = v_i \quad &\hbox{in} \ \ \Omega \\[10pt] (\z_i,Du_i) = \vert Du_i \vert \quad &\hbox{as measures}\\[10pt] [\z_i,\nu_\Omega] = \psi \quad &\hbox{in} \ \ \Gamma_N \\[10pt] [\z_i,\nu_\Omega] \in \hbox{sign}(g-u_i) \quad &\hbox{in} \ \ \Gamma_D.   \end{array}\right.
  	\end{equation}
  	Therefore, for every Borel set $B \subset \Omega$ we have
  	$$\int_{B} (\z_1 - \z_2, Du_1 - Du_2) =  \int_{B} \vert Du_1 \vert - \int_{B} (\z_1, Du_2) +  \int_{B} \vert Du_2 \vert_\nu - \int_{B} (\z_2, Du_1) \geq 0. $$
  	Hence, by equation \eqref{410},
  	$$ \int_{B} \theta(\z_1 - \z_2, D (u_1 - u_2), x ) \, d \vert D(u_1 - u_2) \vert = \int_{B} (\z_1 - \z_2, D (u_1 - u_2)) \geq 0$$
  	for all Borel sets $B \subset \Omega$. Thus,
  	$$\theta(\z_1 - \z_2, D (u_1 - u_2), x ) \geq 0 \quad \vert D(u_1 - u_2) \vert\hbox{-a.e. on} \ \Omega.$$
  	Moreover, since $|DT(u_1 - u_2)|$ is absolutely continuous with respect to $|D(u_1-u_2)|$, we also have
  	$$\theta(\z_1 - \z_2, D (u_1 - u_2), x ) \geq 0 \quad \vert DT(u_1 - u_2) \vert\hbox{-a.e. on} \ \Omega.$$
  	Then, applying the Green formula, we have
  	$$ \int_{\Omega} T(u_1-u_2)(v_1-v_2)\, dx =  \int_{\Omega} T(u_1-u_2) (\divi \z_2 - \divi \z_1) dx$$  $$=  \int_{\Omega} (\z_1 - \z_2, DT(u_1-u_2))+ \int_{\partial \Omega} [\z_2 - \z_1,\nu_\Omega] T(u_1-u_2) \, d\mathcal{H}^{N-1}.$$
  	Now, since
  	$$\int_{\Omega} (\z_1 - \z_2, DT(u_1-u_2))
  	= \int_{\Omega} \theta(\z_1 - \z_2, D(u_1 - u_2), x ) \, d \vert  DT(u_1 - u_2) \vert \geq 0,$$
  	we only need to show that
  	\begin{equation}\label{Intefrate}\int_{\partial \Omega} [\z_2 - \z_1,\nu_\Omega] T(u_1-u_2) \, d\mathcal{H}^{N-1} = \int_{\Gamma_D} [\z_2 - \z_1,\nu_\Omega] T(u_1-u_2) \, d\mathcal{H}^{N-1}\geq 0.
  	\end{equation}
  	To do this, let us consider several cases depending on the values of  $u_1$ and $ u_2$ at a point $x \in \Gamma_D$:

  	\begin{enumerate}
  		\item $u_1(x) < g(x)$ and $u_2(x) < g(x)$: then, $[\z_2 - \z_1,\nu_\Omega](x) = 0$, so the integrand in \eqref{Intefrate} equals zero. A similar argument works whenever $u_1(x) > g(x)$ and $u_2(x) > g(x)$.

  		\item $u_1(x) < g(x) < u_2(x)$: then, $[\z_1,\nu_\Omega](x) =1$ and $[\z_2,\nu_\Omega](x) =- 1$, so $[\z_2 - \z_1,\nu_\Omega](x) =-2 $. By our assumptions on $T$, we have that $T(u_1(x) - u_2(x)) \leq0$, so the integrand in \eqref{Intefrate} is nonnegative. A similar argument works if $u_1(x) > g(x) > u_2(x)$.

  		\item $u_1(x) < g(x) = u_2(x)$: then, $[\z_1,\nu_\Omega](x) =1$ and $[\z_2,\nu_\Omega](x) \in [-1, 1]$, so $[\z_2 - \z_1,\nu_\Omega](x) \leq 0$. By our assumptions on $T$, we have that $T(u_1(x) - u_2(x)) = T(u_1(x) - g(x)) \leq 0$, so the integrand in \eqref{Intefrate} is nonnegative. A similar argument works whenever $u_1(x) > g(x) = u_2(x)$.

  		\item $u_1(x) = g(x) < u_2(x)$: then, $[\z_1,\nu_\Omega](x) \in [-1,1]$ and $[\z_2,\nu_\Omega](x) = -1$, so $[\z_2 - \z_1,\nu_\Omega](x) \leq 0$. By our assumptions on $T$, we have that $T(u_1(x) - u_2(x)) = T(g(x) - u_2(x)) \leq 0$, so the integrand in \eqref{Intefrate} is nonnegative. A similar argument works whenever $u_1(x) = g(x) > u_2(x)$.

  		\item $u_1(x) = u_2(x) = g(x)$: then, $T(u_1(x) - u_2(x)) = T(0) = 0$, so the integrand in \eqref{Intefrate} equals zero.
  	\end{enumerate}
  	We covered all the cases depending on the relative positions of $u_1(x), u_2(x)$ and $g(x)$, so the integrand in \eqref{Intefrate} is always nonnegative; we integrate over $\partial\Omega$ to conclude the proof of the claim \eqref{Intefrate}.
  \end{proof}

 \begin{proof}[Proof of Theorem \ref{ThConpAccre}]   By Theorem \ref{ThStat1} and Lemma \ref{CAoperator} we have that $ {-\Delta }^{\psi,g}_1$ is $m$-completely accretive in $L^2(\Omega)$. Then, by Lemma \ref{contenido} and maximal accretivity, we have $ {-\Delta }^{\psi,g}_1 = \partial_{L^2(\Omega)} \mathcal{F}_{\psi,g}$. Finally, by  \cite[Proposition 2.11]{BrezisBook}, we have
$$ D( {-\Delta }^{\psi,g}_1) = D(\partial \mathcal{F}_{\psi,g}) \subset  D(\mathcal{F}_{\psi,g}) =  BV(\Omega) \cap L^2(\Omega) \subset \overline{D(\mathcal{F}_{\psi,g})}^{L^2(\Omega)} = \overline{D(\partial \mathcal{F}_{\psi,g})}^{L^2(\Omega)}.$$
Therefore, the domain of $ {-\Delta }^{\psi,g}_1$ is dense in $L^2(\Omega)$.
 \end{proof}

 \begin{proof}[Proof of Theorem \ref{inteo310}]
This result
follows directly using Brezis-Komura Theorem (Theorem \ref{Bre-Kom}) and the complete accretivity of the operator $ {-\Delta }^{\psi,g}_1$.
 \end{proof}

\subsection{Some properties of the operator 	$ {-\Delta }^{\psi,g}_1$ }

First, let us  give the proof of the proposition which gives equivalent formulations for the solutions of   problem~\eqref{ProblemIstat}.

\begin{proof}[Proof of Proposition \ref{characterisationI}]  (a) $\Rightarrow$ (d): Multiplying the equation $v = -\mathrm{div}(\z)$ by $w - u$,  integrating over $\Omega$, and using the Green formula, we get
	\begin{align}
		\int_{\Omega} v (w-u) \, dx&= - \int_{\Omega} (w - u) \, \mathrm{div}(\z) \, dx \\
		&= \int_{\Omega} (\z,Dw) - \int_{\partial\Omega} [\z,\nu_\Omega] \,  w \,  d\mathcal{H}^{N-1} \\
		\qquad\qquad\qquad &-  \int_{\Omega} (\z,Du) +\int_{\partial\Omega} [\z,\nu_\Omega] \,  u \,  d\mathcal{H}^{N-1}\\
		&= \int_{\Omega} (\z,Dw) - \int_{\Gamma_D}  [\z,\nu_\Omega]( w - g) \, d\mathcal{H}^{N-1}- \int_{\Gamma_N} [\z,\nu_\Omega]w \, d\mathcal{H}^{N-1} \\
		\qquad\qquad\qquad - &\int_{\Omega} (\z,Du) + \int_{\Gamma_D}  [\z,\nu_\Omega]  (u - g)  \, d\mathcal{H}^{N-1} + \int_{\Gamma_N} [\z,\nu_\Omega]u \, d\mathcal{H}^{N-1}\\ &=\int_{\Omega} (\z,Dw) - \int_{\Gamma_D}  [\z,\nu_\Omega]( w - g) \, d\mathcal{H}^{N-1}- \int_{\Gamma_N} \psi w \, d\mathcal{H}^{N-1}  \\
		\qquad\qquad\qquad - &\int_{\Omega} |Du| - \int_{\Gamma_D} | u - g| \,  d\mathcal{H}^{N-1} + \int_{\Gamma_N} \psi u \, d\mathcal{H}^{N-1}.
	\end{align}
	
	\noindent (d) $\Rightarrow$ (b): trivial.
	
	Since
	$$- \int_{\Gamma_D} [\z,\nu_\Omega] \, (w - g) \, d\mathcal{H}^{N-1} \leq \int_{\Gamma_D}\vert w - g \vert \, d\mathcal{H}^{N-1},$$
	we have that (b) $\Rightarrow$ (c).
	
	\noindent (c) $\Rightarrow$ (b): By assumption (c), there exists a vector field $\z \in X_2(\Omega)$ with  $\| \z \|_\infty \leq 1$  satisfying \eqref{divmed} and \eqref{eq:variationalequalitydirichletprN}.  Let us define  $\varphi \in L^1(\partial \Omega)$ as
	$$\varphi(x)= \left\{  \begin{array}{ll} g(x) \quad &\hbox{if} \ \ x \in \Gamma_D, \\[10pt] 0 \quad &\hbox{if} \ \ x \in \Gamma_N.\end{array} \right.$$
	Given $w \in BV(\Omega) \cap L^2(\Omega)$, by \cite[Theorem B.3, Lemma C.1]{ACMBook} there exists $w_n \in BV(\Omega) \cap L^\infty(\Omega)$ such that $w_n \to w$ in $L^2(\Omega)$ and $Tr(w_n) = \varphi$ for all $n \in \N$. Then,
	taking $w = w_n$ in \eqref{eq:variationalequalitydirichletprN}, we get
	
	\begin{align}
		&\int_{\Omega} v(w_n-u) \, dx \le \int_{\Omega} (\z,Dw_n) - \int_{\Omega} |Du|
		\\
		&\qquad\qquad  - \int_{\Gamma_D} | u - g| \, d\mathcal{H}^{N-1}-  \int_{\Gamma_N} \psi (- u)\, d\mathcal{H}^{N-1}.
	\end{align}
	Applying Green's formula, we have
	\begin{align}
		&\int_{\Omega} v(w_n-u) \, dx \le - \int_{\Omega} {\rm div}(\z)w_n \, dx + \int_{\Gamma_D} [\z,\nu_\Omega] \,  g \,  d\mathcal{H}^{N-1}  - \int_{\Omega}  |Du|
		\\
		&\qquad\qquad  - \int_{\Gamma_D} | u - g| \, d\mathcal{H}^{N-1}-  \int_{\Gamma_N} \psi (- u)\, d\mathcal{H}^{N-1}.
	\end{align}
	Letting $n \to \infty$, it follow that
	\begin{align}
		&\int_{\Omega} v(w-u) \, dx \le - \int_{\Omega} {\rm div}(\z)w \, dx + \int_{\Gamma_D} [\z,\nu_\Omega] \,  g \,  d\mathcal{H}^{N-1} - \int_{\Omega}  |Du|
		\\
		&\qquad\qquad  - \int_{\Gamma_D} | u - g| \, d\mathcal{H}^{N-1}-  \int_{\Gamma_N} \psi (- u)\, d\mathcal{H}^{N-1}.
	\end{align}
	The, applying Green's formula, we get
	\begin{align}
		&\int_{\Omega} v(w-u) \, dx \le \int_{\Omega} (\z,Dw_n) -  \int_{\Gamma_D} [\z,\nu_\Omega] \,  w \,  d\mathcal{H}^{N-1} -  \int_{\Gamma_N} [\z,\nu_\Omega] \,  w \,  d\mathcal{H}^{N-1}
		\\
		&\qquad\qquad + \int_{\Gamma_D} [\z,\nu_\Omega] \,  g \,  d\mathcal{H}^{N-1} - \int_{\Omega}  |Du| - \int_{\Gamma_D} | u - g| \, d\mathcal{H}^{N-1}-  \int_{\Gamma_N} \psi (- u)\, d\mathcal{H}^{N-1}.
	\end{align}
	Thus, (b) holds.
	
	\noindent (b) $\Rightarrow$ (a): If we take $w = u$ in ~\eqref{eq:variationalequalitydirichletpr} and reorganise the terms, we get
	\begin{equation}\label{eq:variationalequalitydirichletpr2}
		\int_{\Omega} |Du| + \int_{\Gamma_D} | u - g| \,d\mathcal{H}^{N-1} \leq \int_{\Omega} (\z,Du) + \int_{\Gamma_D} [\z,\nu_\Omega] ( g - u) \, d\mathcal{H}^{N-1}.
	\end{equation}
	Since $\| \z \|_\infty \leq 1$ we have
	\begin{equation}\label{eq:variationalequalitypart1}
		\int_\Omega (\z,Du) \leq \int_\Omega |Du|;
	\end{equation}
	and, since $\| \z \|_\infty \leq 1$, we also have $\| [\z,\nu_\Omega]\|_\infty \leq 1,$ so
	\begin{equation}\label{eq:variationalequalitypart2}
		\int_{\Gamma_D} [\z,\nu_\Omega]( g-u) \, d\mathcal{H}^{N-1}\le  \int_{\Gamma_D} |u - g| \, d\mathcal{H}^{N-1}.
	\end{equation}
	Hence, in inequality \eqref{eq:variationalequalitydirichletpr2} we actually have an equality. But this implies that also \eqref{eq:variationalequalitypart1} and \eqref{eq:variationalequalitypart2} are also equalities, and this implies
	$$(\z,Du) = \vert Du \vert \quad \hbox{as measures}$$
	and
	$$ [\z,\nu_\Omega] \in \hbox{sign}(g-u) \quad \hbox{in} \ \ \Gamma_D.$$
\end{proof}

We end this subsection with the following result for~$ {-\Delta }^{\psi,g}_1$ that we will use in the next section in the proof of Lemma~\ref{general}, and which is interesting by itself (see also Remark~\ref{enlinfbutpost},  Lemma~\ref{nteo317} and Lemma~\ref{nteo317c}}).

\begin{proposition}\label{bounded1} Assume that $\psi \in L^\infty(\Gamma_N)$ with  $\Vert \psi \Vert_\infty \le 1$ and $g \in L^1(\Gamma_D)$, and let $f \in L^2(\Omega)$. If $u \in BV(\Omega)$  is a solution of
	$$f \in u + {\lambda}  ({-\Delta }^{\psi,g}_1)(u) \quad  \hbox{with} \ \lambda >0,$$ then, for $C_\Omega$ the constant in Theorem~\ref{pepesab},
	$$\left\|u\right\|_2\le \left\|\,|f|+ { \lambda }C_\Omega\,\right\|_2.$$
	If $f \in L^\infty (\Omega)$, then
	$$\left\|  u\right\|_\infty\le \left\|  f\right\|_\infty+ { \lambda }C_\Omega.$$
\end{proposition}
\begin{proof} By Proposition \ref{characterisationI}, we have and there exists a vector field $\z \in X_2(\Omega)$ with  $\| \z \|_\infty \leq 1$  satisfying
	\begin{align}\label{eq:variationalequalityBound}
		&{\frac{1}{\lambda}}\int_{\Omega} (f -u)(w-u) \, dx \le \int_{\Omega} (\z,Dw) - \int_{\Omega} |Du|
		\\
		&\qquad\qquad + \int_{\Gamma_D}\vert w - g \vert \, d\mathcal{H}^{N-1} - \int_{\Gamma_D} | u - g| \, d\mathcal{H}^{N-1}-  \int_{\Gamma_N} \psi (w- u)\, d\mathcal{H}^{N-1},
	\end{align}
	for every $w \in BV(\Omega) \cap L^2(\Omega).$
	
	Taking $w=0$ in \eqref{eq:variationalequalityBound} we have
	$${\frac{1}{\lambda}} \int_{\Omega} u^2  \, dx \le { \frac{1}{\lambda}} \int_{\Omega} fu \, dx - \int_{\Omega} |Du|
	+ \int_{\Gamma_D}\vert  g \vert \, d\mathcal{H}^{N-1} - \int_{\Gamma_D} | u - g| \, d\mathcal{H}^{N-1}+  \int_{\Gamma_N} \psi u \, d\mathcal{H}^{N-1}$$ $$ \leq { \frac{1}{\lambda}}  \int_{\Omega} fu  \, dx  - \int_{\Omega} |Du|
	+ \int_{\Gamma_D}|u| d\mathcal{H}^{N-1} +  \int_{\Gamma_N} \psi u \, d\mathcal{H}^{N-1}$$
	$$ \leq  {\frac{1}{\lambda}} \int_{\Omega} fu  \, dx  - \int_{\Omega} |Du|
	+ \int_{\partial\Omega}| u| \, d\mathcal{H}^{N-1}.$$
	Now, by Theorem \ref{pepesab}, we have
	$$ \int_{\partial\Omega} |u| \, d\mathcal{H}^{N-1} \leq    \int_{\Omega} |Du| + C_\Omega\int_\Omega \vert u \vert \, dx.$$
	Hence,
	$${ \frac{1}{\lambda}}\int_{\Omega} u^2  \, dx \le { \frac{1}{\lambda}}\int_{\Omega} fu  \, dx + C_\Omega   \int_\Omega \vert u \vert \, dx = { \frac{1}{\lambda} \int_\Omega |u| (|f| + \lambda C_\Omega) dx,}$$
	so  by Young's inequality, we have
	$$\left\|u\right\|_2\le \left\|\,|f|+\lambda C_\Omega\,\right\|_2.$$
	
	Assume now that $f \in L^\infty(\Omega)$. Given $m \in \N$, $m\ge 2$, taking $w=u - \vert u \vert^{m-2}u$ in \eqref{eq:variationalequalityBound} (truncate if necessary), we have (working as above)
	$$\begin{array}{c}\displaystyle
		\frac{1}{\lambda}\int_{\Omega} \vert u \vert^m  \, dx \le { \frac{1}{\lambda}}\int_{\Omega} \vert u \vert^{m-2}u f \, dx - \int_{\Omega} |D\vert u \vert^{m-2}u|
		\\[12pt]\displaystyle + \int_{\Gamma_D}\vert u\vert^{m-1}  \, d\mathcal{H}^{N-1}+\int_{\Gamma_N} \vert u\vert^{m-1}  \, d\mathcal{H}^{N-1}\\[12pt]\displaystyle \le  \int_{\Omega} \vert u \vert^{m-1}( |f|+ { \lambda}C_\Omega) \, dx
		\\[12pt]\displaystyle  \le \left(\int_{\Omega} \vert u \vert^{m} \, dx\right)^{\frac{m-1}{m}}\left(\int_{\Omega}\left(|f|+ { \lambda}C_\Omega\right)^m\right)^{\frac{1}{m}}.
	\end{array}$$
	Hence,
	$${ \left(\int_{\Omega}|u|^m dx \right)^{\frac{1}{m}}}\le \left(\int_{\Omega}\left(|f|+\lambda C_\Omega\right)^m\right)^{\frac{1}{m}},$$
	from where
	$$\left\|  u\right\|_\infty\le \left\|  f\right\|_\infty+\lambda C_\Omega.$$
\end{proof}

\section{Proofs for  the case   $||\psi||_\infty \leq  1$.}\label{sectionla5}

 Let us begin by giving a characterization of the operator  $ {-\widetilde \Delta }^{\psi,g}_1$, that can be obtained with a similar proof to the one of Proposition \ref{characterisationI}.

\begin{proposition}\label{characterisationIK} Let $\psi \in L^\infty(\Gamma_N)$ be such that  $\Vert  \psi \Vert_\infty \le 1$ and  $g \in L^1(\Gamma_D)$.
     The following conditions are equivalent:
\\
(a) $(u,v) \in  {-\widetilde \Delta }^{\psi,g}_1$;
\\
(b)  $T_k(u) \in  BV(\Omega)\cap L^2(\Omega)$, $v\in L^2(\Omega)$, and there exists a vector field $\z \in X_2(\Omega)$ with  $\| \z \|_\infty \leq 1$  such that
\begin{equation}\label{divmedK}\begin{array}{c}\displaystyle-\mathrm{div}(\z) = v \ \hbox{ in } \  \mathcal{D}'(\Omega),\\[12pt]
\displaystyle [\z,\nu_\Omega] = \psi \ \hbox{ in } \Gamma_N,
\end{array}\end{equation}
and the following variational inequality holds true:
\begin{align}\label{eq:variationalequalitydirichletprK}
&\int_{\Omega} v(w-T_k(u)) \, dx \le \int_{\Omega} (\z,Dw) - \int_{\Omega} |DT_k(u)|
\\
& - \int_{\Gamma_D} [\z,\nu_\Omega] \, (w - T_k(g)) \, d\mathcal{H}^{N-1} - \int_{\Gamma_D} |T_k(u)- T_k(g)| \, d\mathcal{H}^{N-1}
\\
&
-  \int_{\Gamma_N} \psi (w- T_k(u))\, d\mathcal{H}^{N-1},
\end{align}
for every $w \in BV(\Omega) \cap L^2(\Omega)$;
\\
(c) $T_k(u) \in  BV(\Omega)\cap L^2(\Omega)$, $v\in L^2(\Omega)$, and there exists a vector field $\z \in X_2(\Omega)$ with  $\| \z \|_\infty \leq 1$  satisfying \eqref{divmed} and the following variational inequality holds true:
\begin{align}\label{eq:variationalequalitydirichletprNK}
&\int_{\Omega} v(w-T_k(u)) \, dx \le \int_{\Omega} (\z,Dw) - \int_{\Omega} |DT_k(u)|
\\
& + \int_{\Gamma_D}\vert w - T_k(g) \vert \, d\mathcal{H}^{N-1} - \int_{\Gamma_D} | T_k(u) - T_k(g)| \, d\mathcal{H}^{N-1}-  \int_{\Gamma_N} \psi (w- T_k(u))\, d\mathcal{H}^{N-1},
\end{align}
for every $w \in BV(\Omega) \cap L^2(\Omega)$;
\\
(d)  $T_k(u) \in  BV(\Omega)\cap L^2(\Omega)$, $v\in L^2(\Omega)$, and there exists a vector field $\z \in X_2(\Omega)$ with  $\| \z \|_\infty \leq 1$  satisfying \eqref{divmed}
and  the following variational equality holds true:
\begin{align}\label{eq:variationalequalitydirichletK}
&\int_{\Omega} v(w-T_k(u)) \, dx = \int_{\Omega} (\z,Dw) - \int_{\Omega} |DT_k(u)|
\\
& - \int_{\Gamma_D} [\z,\nu_\Omega] \, (w - T_k(g)) \, d\mathcal{H}^{N-1} - \int_{\Gamma_D} | T_k(u) - T_k(g)| \, d\mathcal{H}^{N-1}
\\
&
-  \int_{\Gamma_N} \psi ( w-T_k(u))\, d\mathcal{H}^{N-1},
\end{align}
for every $w \in BV(\Omega) \cap L^2(\Omega)$.
\end{proposition}

 We have the following comparison result.
 \begin{lemma}\label{comparison1} Let $\psi_1, \psi_2 \in  L^\infty(\Gamma_N)$ be such that  $\Vert  \psi_i \Vert_\infty \leq 1$, $i=1,2$, and $ g_1,g_2 \in L^1(\Gamma_D)$. Let $u_1, u_2$ satisfying
$$   f_i \in u_i - \widetilde \Delta^{\psi_i,g_i}_1(u_i), \quad i=1,2, \quad f_i \in L^2(\Omega).$$
Then, if $f_1 \leq f_2$ $\mathcal{L}^{N}$-a.e. in $\Omega$, $g_1 \leq g_2$ $\mathcal{H}^{N-1}$-a.e. in $\Gamma_D$, and $\psi_1 \leq \psi_2$ $\mathcal{H}^{N-1}$-a.e. in~$\Gamma_N$, we have $u_1 \leq u_2$ $\mathcal{L}^N$-a.e. in $\Omega$.
\end{lemma}
\begin{proof} By Proposition \ref{characterisationIK} we have there exist a vector field $\z_i \in X_2(\Omega)$ with  $\| \z_i \|_\infty \leq 1$  such that
$$\begin{array}{c}\displaystyle-\mathrm{div}(\z_i) =  f_i - u_i \ \hbox{ in } \  \mathcal{D}'(\Omega),  \ \  i =1,2,\\[8pt]
\displaystyle [\z_i,\nu_\Omega] = \psi_i \ \hbox{ in } \Gamma_N,
\end{array}$$
and the following variational inequality holds true:
\begin{align}\label{eq:variationalequalitydirichletprNN}
&\int_{\Omega} ( f_i - u_i)(w-T_k(u_i)) \, dx \le \int_{\Omega} (\z_i,Dw) - \int_{\Omega} |DT_k(u_i)|
\\
&- \int_{\Gamma_D} [\z_i,\nu_\Omega] \, (w - T_k(g_i)) \, d\mathcal{H}^{N-1} - \int_{\Gamma_D} |T_k(u_i) - T_k(g_i))| \, d\mathcal{H}^{N-1}
\\
&
-  \int_{\Gamma_N} \psi_i (w-T_k(u_i)), d\mathcal{H}^{N-1},
\end{align}
for every $w \in BV(\Omega) \cap L^2(\Omega)$. Taking in \eqref{eq:variationalequalitydirichletprNN} $w := T_k(u_1)- (T_k(u_1) - T_k(u_2))^+$  for $i=1$, and $w:= T_k(u_2) + (T_k(u_1) - T_k(u_2))^+$ for $i=2$, we obtain
\begin{equation}\label{eq:variationalequalitydirichletprNNN1}
\begin{array}{lll}
\displaystyle-\int_{\Omega} ( f_1 - u_1)(T_k(u_1) - T_k(u_2))^+ \, dx \le -\int_{\Omega} (\z_1,D(T_k(u_1) - T_k(u_2))^+) \\[12pt]
- \displaystyle \int_{\Gamma_D} [\z_1,\nu_\Omega] \, (T_k(u_1)- (T_k(u_1) - T_k(u_2))^+ -  T_k(g_1)) \, d\mathcal{H}^{N-1} \\[12pt] - \displaystyle\int_{\Gamma_D} | T_k(u_1)- T_k(g_1)| \, d\mathcal{H}^{N-1}+ \int_{\Gamma_N} \psi_1(T_k(u_1) - T_k(u_2))^+ \, d\mathcal{H}^{N-1},\end{array}
\end{equation}
and
\begin{equation}\label{eq:variationalequalitydirichletprNNN2}\begin{array}{lll}
\displaystyle\int_{\Omega} ( f_2 - u_2)(T_k(u_1) - T_k(u_2))^+ \, dx \le \int_{\Omega} (\z_2,D(T_k(u_1) - T_k(u_2))^+)
\\[12pt]  \displaystyle - \int_{\Gamma_D} [\z_2,\nu_\Omega] \, (T_k(u_2) + (T_k(u_1) - T_k(u_2))^+ -  T_k(g_2)) \, d\mathcal{H}^{N-1} \\[12pt] - \displaystyle\int_{\Gamma_D} | T_k(u_2)-  T_k(g_2)| \, d\mathcal{H}^{N-1}- \int_{\Gamma_N} \psi_2 (T_k(u_1) - T_k(u_2))^+ \, d\mathcal{H}^{N-1}.\end{array}
\end{equation}
We have,
$$- \displaystyle \int_{\Gamma_D} [\z_1,\nu_\Omega] \, (T_k(u_1)- (T_k(u_1) - T_k(u_2))^+ -  T_k(g_1)) \, d\mathcal{H}^{N-1}   - \displaystyle\int_{\Gamma_D} | T_k(u_1)-  T_k(g_1)| \, d\mathcal{H}^{N-1}$$ $$ =  \int_{\Gamma_D} [\z_1,\nu_\Omega] \, (T_k(u_1) - T_k(u_2))^+ \, d\mathcal{H}^{N-1} - \int_{\Gamma_D} [\z_1,\nu_\Omega] \, (T_k(u_1)- T_k(g_1)) \, d\mathcal{H}^{N-1} $$ $$- \displaystyle\int_{\Gamma_D} | T_k(u_1)- T_k( g_1)| \, d\mathcal{H}^{N-1}.$$
Now, since $[\z_1,\nu_\Omega] \in {\rm sign}(T_k(u_1)- T_k( g_1))$, we have
$$- \displaystyle \int_{\Gamma_D} [\z_1,\nu_\Omega] \, (T_k(u_1)- (T_k(u_1) - T_k(u_2))^+ -  T_k(g_1)) \, d\mathcal{H}^{N-1}   - \displaystyle\int_{\Gamma_D} | T_k(u_1)-  T_k(g_1)| \, d\mathcal{H}^{N-1} $$ $$= \int_{\Gamma_D} [\z_1,\nu_\Omega] \, (T_k(u_1) - T_k(u_2))^+ \, d\mathcal{H}^{N-1}.$$
Similarly
$$ \displaystyle - \int_{\Gamma_D} [\z_2,\nu_\Omega] \, (T_k(u_2) + (T_k(u_1) - T_k(u_2))^+ -  T_k(g_2)) \, d\mathcal{H}^{N-1} $$ $$ = \int_{\Gamma_D} [\z_2,\nu_\Omega] \, (T_k(u_1) - T_k(u_2))^+ \, d\mathcal{H}^{N-1}. $$
Then, adding \eqref{eq:variationalequalitydirichletprNNN1} and \eqref{eq:variationalequalitydirichletprNNN2}, we get
$$\displaystyle\int_{\Omega} ( u_1 -u_2)(T_k(u_1)- (T_k(u_1) - T_k(u_2))^+  \, dx  $$ $$\le \int_{\Omega} (f_1 -f_2)(T_k(u_1)- (T_k(u_1) - T_k(u_2))^+  \, dx  -\int_{\Omega} (\z_1- \z_2,D(T_k(u_1)- (T_k(u_1) - T_k(u_2))^+ ) $$ $$- \int_{\Gamma_D} ([\z_1,\nu_\Omega] -[\z_2,\nu_\Omega]) \, (T_k(u_1)- (T_k(u_1) - T_k(u_2))^+  \, d\mathcal{H}^{N-1}$$ $$+ \int_{\Gamma_N} (\psi_1 - \psi_2)(T_k(u_1)- (T_k(u_1) - T_k(u_2))^+  \, d\mathcal{H}^{N-1}.$$
On the other hand, by Proposition \ref{prop:chain-rule}, we have
$$\int_{\Omega} (\z_1- \z_2,D((T_k(u_1)- (T_k(u_1) - T_k(u_2))^+ ) = \int_{\Omega} (\z_1- \z_2,D(T_k(u_1) - T_k(u_2)) \geq 0.$$
And,  by the assumptions,
$$\int_{\Gamma_D} ([\z_1,\nu_\Omega] -[\z_2,\nu_\Omega]) \, (u_1 - u_2)^+ \, d\mathcal{H}^{N-1} \geq 0,$$
$$   \int_{\Omega} (f_1 -f_2) (T_k(u_1) - T_k(u_2))^+ \, dx\le 0$$
and
$$ \int_{\Gamma_N} (\psi_1 - \psi_2) (T_k(u_1) - T_k(u_2))^+ \, d\mathcal{H}^{N-1} \leq 0.$$
Consequently, letting $k \to \infty$, we have
$$\displaystyle\int_{\Omega} ( u_1 -u_2)(u_1 - u_2)^+ \, dx \leq 0,$$
therefore $u_1 \leq u_2$ $\mathcal{L}^N$-a.e. in $\Omega$.
\end{proof}

 From the above proof we can obtain that $ -\widetilde \Delta^{\psi,g}_1$ is accretive.

\begin{lemma}\label{general} Let $\psi \in L^\infty(\Gamma_N)$ be such that  $\Vert  \psi \Vert_\infty \leq 1$ and  $g \in L^1(\Gamma_D)$.  We have that the operator  $ -\widetilde \Delta^{\psi,g}_1$ is $m$-completely accretive   (hence maximal monotone) in $L^2(\Omega)$.
\end{lemma}
\begin{proof}   Let us first see that $ -\widetilde \Delta^{\psi,g}_1$ is completely accretive. By Proposition \ref{prop:completely-accretive}, we need to show that
\begin{equation}\label{CAcret} \int_{\Omega} q(u_1-u_2)(v_1-v_2)\, dx \geq 0 \end{equation}
for every $q\in P_{0}$ and every $(u_i,v_i) \in  -\widetilde \Delta^{\psi_i,g_i}_1$,   $i=1,2$ .

Now, if $(u_i, v_i) \in  -\widetilde \Delta^{\psi_i,g_i}_1$,   $i =1,2$, then  $T_k(u_i) \in  BV(\Omega)$ for all $k>0$ and there exists $\z_i \in X_2(\Omega)$ with $\Vert \z_i \Vert_\infty \leq 1$ satisfying:
\begin{equation}\label{E1FundNCANew}
\left\{ \begin{array}{llll} - \divi \z_i = v_i \quad &\hbox{in} \ \ \Omega, \\[12pt] (\z_i,DT_k(u_i)) = \vert DT_k(u_i) \vert \quad &\hbox{as measures},\\[12pt] [\z_i,\nu_\Omega] = \psi \quad &\hbox{in} \ \ \Gamma_N , \\[12pt] [\z_i,\nu_\Omega] \in \hbox{sign}(T_k(g)-T_k(u_i)) \quad &\hbox{in} \ \ \Gamma_D.   \end{array}\right.
\end{equation}
From here, for every Borel set $B \subset \Omega$, we have
$$\int_B (\z_1 - \z_2, DT_k(u_1) - DT_k(u_2)) $$ $$=  \int_B  \vert DT_k(u_1)\vert - \int_B (\z_1,DT_k(u_2)) + \int_B  \vert DT_k(u_2)\vert  - \int_B (\z_2,DT_k(u_1)) \geq 0.$$
Hence, by \eqref{410},
$$\begin{array}{c}\displaystyle
\int_B \theta(\z_1 - \z_2, D(T_k(u_1) T_k(u_2)),x) d \vert D(T_k(u_1) - T_k(u_2) \vert \\[12pt]
\displaystyle = \int_B (\z_1 - \z_2, DT_k(u_1) - DT_k(u_2)) \geq 0.
\end{array}$$
Thus
$$\theta(\z_1 - \z_2, D(T_k(u_1) T_k(u_2)),x) \geq 0 \quad  \vert D(T_k(u_1) - T_k(u_2) \vert-\hbox{a.e on} \ \Omega.$$
Applying Proposition \ref{prop:chain-rule} we get that
$$\theta(\z_1 - \z_2, D(T_k(u_1) T_k(u_2)),x) =  \theta(\z_1 - \z_2, DT(T_k(u_1) T_k(u_2)),x)$$
a.e. with respect to the measures $\vert D(T_k(u_1) T_k(u_2) \vert$ and $\vert  D(T_k(u_1) T_k(u_2)) \vert.$ We then conclude that
\begin{equation}\label{positivity1}\theta(\z_1 - \z_2, D(T_k(u_1) T_k(u_2)),x) \geq 0 \quad \vert D(T_k(u_1) T_k(u_2) \vert-\hbox{a.e. on} \ \ \Omega.\end{equation}

Applying Green's formula, and having in mind \eqref{positivity1}, we have,  for  $q\in P_{0}$,
$$\begin{array}{c}
\displaystyle \int_{\Omega} q(T_k(u_1) - T_k(u_2))(v_1 - v_2)\, dx = \int_{\Omega} {\rm div}(\z_2 - \z_1) q(T_k(u_1) - T_k(u_2)) \, dx
\\[12pt] \displaystyle= \int_\Omega (\z_1 - \z_2, Dq(T_k(u_1) - T_k(u_2)) + \int_{\partial \Omega} [\z_1 - \z_2, \nu_\Omega]  q(T_k(u_1) - T_k(u_2)) \,  d\mathcal{H}^{N-1} =  \\[12pt]
 \displaystyle
\geq \int_{\partial \Omega} [\z_1 - \z_2, \nu_\Omega]  Tq(T_k(u_1) - T_k(u_2)) \,  d\mathcal{H}^{N-1}
\\[12pt] \displaystyle
= \int_{\Gamma_D} [\z_1 - \z_2, \nu_\Omega]  q(T_k(u_1) - T_k(u_2)) \,  d\mathcal{H}^{N-1}.
\end{array}$$
Now, considering several cases dependding on the values of $u_1$ and $u_2$ at the points of $\Gamma_D$, as we did in the proof of Theorem \ref{CHARAT}, we have that
$$\int_{\Gamma_D} [\z_1 - \z_2, \nu_\Omega]  q(T_k(u_1) - T_k(u_2)) \,  d\mathcal{H}^{N-1} \geq 0.$$
Consequently,
$$ \int_{\Omega} q(T_k(u_1) - T_k(u_2))(v_1 - v_2)\, dx \geq 0 \quad \hbox{for all} \ k>0.$$
Then, taking limits as $k \to \infty$, we get the inequality \eqref{CAcret}, and therefore $\widetilde{\mathcal{A}}_{\psi,g}$ is completely accretive.

  To prove that $ -\widetilde \Delta^{\psi,g}_1$ is $m$-completely accretive in $L^2(\Omega)$,  by   Minty Theorem (Theorem \ref{Minty}), we only need to prove that the following range condition holds:
\begin{equation}\label{CR1New}
\hbox{given} \ f \in L^2(\Omega), \ \ \exists \, u \in D( -\widetilde \Delta^{\psi,g}_1) \ \ \hbox{such that} \ \  f \in u  -\widetilde \Delta^{\psi,g}_1(u).
\end{equation}

 \noindent{\bf Step 1.} Suppose first that there exists $a \in \R$ such that $-1 < a \leq \psi(x)$ for all $x \in \Gamma_N$. For every $n \in \N$, $n\ge 2$, let $\psi_n:= T_{1 - \frac{1}{n}}(\psi)$. Then, since $\Vert \psi_n \Vert_\infty < 1$,  by Theorem~\ref{ThConpAccre} (and Proposition~\ref{characterisationI}) there exists $u_n \in BV(\Omega)$ and $\z_n \in X_2(\Omega)$ with  $\| \z_n \|_\infty \leq 1$  such that
$$\begin{array}{c}\displaystyle-\mathrm{div}(\z_n) = f -u_n \ \hbox{ in } \  \mathcal{D}'(\Omega),\\[12pt]
\displaystyle [\z_n,\nu_\Omega] = \psi_n \ \hbox{ in } \Gamma_N,
\end{array}$$
and the following variational inequality holds true:
\begin{align}\label{newdefeq:variationalequalitydirichletprNew}
&\int_{\Omega} (f -u_n)(w-u_n) \, dx \le \int_{\Omega} (\z_n,Dw) - \int_{\Omega} |Du_n|
\\
&\qquad\qquad - \int_{\Gamma_D} [\z_n,\nu_\Omega] \, (w - g) \, d\mathcal{H}^{N-1} - \int_{\Gamma_D} | u_n - g| \, d\mathcal{H}^{N-1}- \int_{\Gamma_N} \psi_n (w-u_n)\, d\mathcal{H}^{N-1},
\end{align}
for every $w \in BV(\Omega) \cap L^2(\Omega)$.

Since $\Vert [\z_n,\nu_\Omega] \Vert_\infty \leq \| \z_n \|_\infty \leq 1$ for all $n \in \N$, we can assume, taking a subsequence if required,  that \begin{equation}\label{converg1}
\z_n  \rightharpoonup  \z \quad \hbox{weakly$^*$ in} \ L^\infty(\Omega) \quad \hbox{and} \quad [\z_n,\nu_\Omega] \rightharpoonup \z \quad \hbox{weakly$^*$ in} \  L^\infty(\partial \Omega).
\end{equation}

Then, we get
$$\displaystyle-\mathrm{div}(\z) = f -u \ \hbox{ in } \  \mathcal{D}'(\Omega).$$

 By Proposition~\ref{bounded1}, we have $\{ u_n \}$ is bounded in $L^2(\Omega)$, so we can assume that
\begin{equation}\label{converg2}
u_n \rightharpoonup u \quad \hbox{weakly in} \ L^2(\Omega).
\end{equation}
Now, by  Lemma~\ref{comparison1},  we have $u_n \leq u_{n+1}$, thus
\begin{equation}\label{converg3}
u_n \to u \quad \hbox{in} \ L^2(\Omega),
\end{equation}
hence by Proposition~\ref{bounded1} again, we have
\begin{equation}\label{est001}
\left\|u\right\|_2\le \left\|\,|f|+C_\Omega\,\right\|_2.
\end{equation}

  For $\widetilde{w}\in BV(\Omega)\cap L^2(\Omega)$,  taking $w=u_n+\widetilde{w}-T_k(u_n)$ in \eqref{newdefeq:variationalequalitydirichletprNew}, we have
$$ \begin{array}{l}\displaystyle
\int_{\Omega} (f -u_n)( \tilde{w}-T_k(u_n)) \, dx \le \int_{\Omega} (\z_n,D \tilde{w}) - \int_{\Omega} |DT_k(u_n)|
\\[12pt]
\displaystyle - \int_{\Gamma_D} [\z_n,\nu_\Omega] \, (\tilde{w}-  T_k(g)-(T_k(u_n) -   T_k(g))) \, d\mathcal{H}^{N-1} - \int_{\Gamma_N} \psi_n ( \tilde{w}-T_k(u_n))\, d\mathcal{H}^{N-1}\\[12pt] \displaystyle \le \int_{\Omega} (\z_n,D\tilde{w}) - \int_{\Omega} |DT_k(u_n)|
\\[12pt]
\displaystyle  +\int_{\Gamma_D} |\widetilde{w}-  T_k(g)|  \, d\mathcal{H}^{N-1} - \int_{\Gamma_D} |T_k(g)-T_k(u_n)| \, d\mathcal{H}^{N-1} - \int_{\Gamma_N} \psi_n ( \tilde{w}-T_k(u_n))\, d\mathcal{H}^{N-1}.
\end{array}$$
Hence, taking limists in $n$ we get that (observe that monotonicity of $u_n$ also holds on the boundary)
$$\begin{array}{l}\displaystyle
\int_{\Omega} (f -u)( \tilde{w}-T_k(u)) \, dx \le \int_{\Omega} (\z,D\tilde{w}) - \int_{\Omega} |DT_k(u)|
\\[12pt]
 \displaystyle +\int_{\Gamma_D} |\widetilde{w}- T_k(g)|  \, d\mathcal{H}^{N-1} - \int_{\Gamma_D} |T_k(g)-T_k(u)| \, d\mathcal{H}^{N-1} - \int_{\Gamma_N} \psi (\tilde{w}-T_k(u))\, d\mathcal{H}^{N-1}.
\end{array}$$
Therefore,  by Proposition \ref{characterisationIK}, \eqref{CR1New} holds.

\noindent{\bf Step 2}.   For a general $\psi$, if we take $\psi_n=\sup\left\{-1+\frac{1}{n},\psi\right\}$, we have, for all $n\in \N$, $n\ge 2$,  $\psi_n$ verifies the assumption of Step 2. Thus,
\begin{equation}\label{CR1Newn}
\hbox{given} \ f \in L^2(\Omega), \ \ \exists \, u_n \in D(\widetilde{\mathcal{A}}_{\psi_n,g}) \ \ \hbox{such that} \ \  f \in u_n + \widetilde{\mathcal{A}}_{\psi_n,g}(u_n).
\end{equation}
Then, working as in the Step 2,  we can show that the range condition   also holds in this case.
\end{proof}

\begin{remark}\label{enlinfbutpost}\rm
	For $||\psi||_\infty\le 1$, if $f\in L^\infty(\Omega)$ then,   by  Proposition~\ref{bounded1}, the solution $u$ of~\eqref{CR1New} satisfies
	$$\Vert u \Vert_{\infty} \leq \Vert  f   \Vert_\infty + C_\Omega.$$
	Thus, we also have the same estimate for the function $u_n$ obtained in the   Step 2 of the proof of the previous Lemma~\ref{general}. Consequently, we have that if $f\in L^\infty(\Omega)$ and $u$ is solution of $f \in u  -\widetilde \Delta^{\psi,g}_1(u)$, then
	\begin{equation}\label{inftBound}
		\Vert u \Vert_{\infty} \leq \Vert  f   \Vert_\infty + C_\Omega.
	\end{equation}
	Note that with a similar reasoning we get that if $f\in L^\infty(\Omega)$ and $u$ is solution of $f \in u - \lambda \widetilde \Delta^{\psi,g}_1(u)$, then
	\begin{equation}\label{inftBoundN}
		\hbox{ $\Vert u \Vert_{\infty} \leq \Vert  f   \Vert_\infty + \lambda C_\Omega$ \ for  any $\lambda >0$,}
	\end{equation}
	and if $f\in L^2(\Omega)$ and $u$ is solution of $f \in u - \lambda \widetilde \Delta^{\psi,g}_1(u)$, then \samepage
	\begin{equation}\label{inftBound2}
		\hbox{ $\Vert u \Vert_{2} \leq \Vert \vert f \vert + \lambda C_\Omega \Vert_2\le \Vert   f  \Vert_2+\lambda C_\Omega|\Omega|^{1/2}$ \ for any $\lambda >0$.  }
	\end{equation}
	\hfill $\blacksquare$
\end{remark}

\begin{lemma}\label{generaldensedomain} Let $\psi \in L^\infty(\Gamma_N)$ be such that  $\Vert  \psi \Vert_\infty \leq 1$ and  $g \in L^1(\Gamma_D)$.  We have that  $D( -\widetilde \Delta^{\psi,g}_1)$ is dense in $L^2(\Omega)$.
\end{lemma}
\begin{proof}
 Given $f\in BV(\Omega)\cap L^\infty(\Omega)$ let  $u_n$ a solution of
$$f \in u_n - \frac{1}{n} \widetilde \Delta^{\psi,g}_1(u_n).$$
Then, by \eqref{inftBoundN}, we have
\begin{equation}\label{inftBoundNN}
\Vert u_n \Vert_\infty \leq \Vert |  f|+\frac{1}{n} C_\Omega\Vert_\infty \quad \forall \, n \in \N.
\end{equation}
Since  $f\in BV(\Omega)\cap L^\infty(\Omega)$ is dense in $L^2(\Omega)$, let us show that $u_n\to f$ to conclude. From~\eqref{eq:variationalequalitydirichletprK}, taking $w=f$ and having in mind that $u_n\in L^\infty(\Omega)$ we can take $k\to +\infty$  to get \begin{align}\label{eq:variationalequalityBoundrepe}
&n\int_{\Omega} (f -u_n)(f-u_n) \, dx \le \int_{\Omega} (\z,Df) - \int_{\Omega} |Du_n|
\\
&\qquad\qquad + \int_{\Gamma_D}\vert f - g \vert \, d\mathcal{H}^{N-1} - \int_{\Gamma_D} | u_n - g| \, d\mathcal{H}^{N-1}-  \int_{\Gamma_N} \psi (f- u_n)\, d\mathcal{H}^{N-1}.
\end{align}
Hence, by \eqref{eq:7}, we have
\begin{align}\label{eq:variationalequalityBoundrepe2}
&n\int_{\Omega} |f -u_n|^2 \, dx \le \int_{\Omega} |Df|
 + \int_{\Gamma_D}\vert f - g \vert \, d\mathcal{H}^{N-1} -  \int_{\Gamma_N} \psi (f- u_n)\, d\mathcal{H}^{N-1} \\& \le 2\int_{\Omega} |Df|  + \int_{\Gamma_D}\vert f - g \vert \, d\mathcal{H}^{N-1}  + \int_{\Gamma_N}\vert f\vert \, d\mathcal{H}^{N-1} + C_\Omega \int_\Omega \vert u_n \vert dx \leq M, .
\end{align}
where $M$ is a constant independent of $n$ by \eqref{inftBoundNN}. Thus
$$\int_{\Omega} |f -u_n|^2 \, dx \to 0.$$
\end{proof}

 \subsection{From   mild solution to strong solution}

  By Theorem \ref{general},  Lemma~\ref{generaldensedomain}, Theorem~\ref{EUm-accretive}  and Theorem~\ref{Breisresult}, we have the following result,

\begin{lemma}\label{ExistUniq} Let $\psi \in L^\infty(\Gamma_N)$ be such that  $\Vert  \psi \Vert_\infty \leq 1$ and  $g \in L^1(\Gamma_D)$. For any $u_0 \in L^2(\Omega)$ and any $T > 0$, there exists a unique mild  solution of problem~\eqref{ProblemI}. Moreover, the following comparison principle holds: for any $q \in [1,\infty]$, if $u_1, u_2$ are mild solutions for the initial data $u_{1,0}, u_{2,0} \in  L^2(\Omega, \nu) \cap L^q(\Omega, \nu) $ respectively, then
\begin{equation}\label{DCompPrincipleplaplaceBVNew}
\Vert (u_1(t) - u_2(t))^+ \Vert_q \leq \Vert ( u_{1,0}- u_{2,0})^+ \Vert_q.
\end{equation}
Furthermore, if $u_0 \in D(-\tilde \Delta^{\psi,g}_1)$, the mild solution is a strong solution.
\end{lemma}

The next result give us interesting bounds of the mild solutions which will allow us to prove that they are in fact strong solutions.

\begin{lemma}\label{nteo317} Let  $u(t)$ be a mild solution of  problem \eqref{ProblemI}. We have:
   \begin{equation}\label{nteo317inf}\hbox{if $u_0\in L^\infty(\Omega)$ then } \  \left\| u(t)\right\|_\infty \le \left\| u_0\right\|_\infty+C_\Omega t,\end{equation}
   and
   \begin{equation}\label{nteo317Two} \hbox{if $u_0\in L^2(\Omega)$ then} \   \left\| u(t)\right\|_2 \le \left\|   u_0   \right\|_2+ |\Omega|^{1/2}C_\Omega t ,\end{equation}
   being  $C_\Omega$ the constant in Theorem~\ref{pepesab}.
\end{lemma}

\begin{proof}  By Crandall-Ligett's exponential formula \eqref{CRLExp}, we have
$$\lim_{n \to \infty} \left(I-\frac{t}{n} \widetilde \Delta^{\psi_m,g}_1\right)^{-n}u_0 = u(t) \quad \hbox{in} \ L^2(\Omega).$$
 Now by \eqref{inftBoundN}, we have
 $$\left\| (I-\frac{t}{n} \widetilde \Delta^{\psi_m,g}_1)^{-1}u_0\right\|_\infty\le \left\|  u_0\right\|_\infty+\frac{t}{n} C_\Omega.$$ Therefore, recursively we get
$$\left\| (I-\frac{t}{n} \widetilde \Delta^{\psi_m,g}_1)^{-n}u_0\right\|_\infty\le \left\|  u_0\right\|_\infty+n\frac{t}{n} C_\Omega=\left\|  u_0\right\|_\infty+t C_\Omega \quad \forall\, n \in \N,$$
and consequently \eqref{nteo317inf} holds.

The proof of \eqref{nteo317Two} is similar  but using \eqref{inftBound2} instead of \eqref{inftBoundN}.
\end{proof}

\begin{lemma}\label{nteo317c}  Let $\psi \in L^\infty(\Gamma_N)$ be such that  $\Vert  \psi \Vert_\infty \leq 1$ and  $g \in L^1(\Gamma_D)$.
 For every $u_0\in L^2(\Omega)$,  the mild solution $u$ of problem~\eqref{ProblemI}  is a strong solution.
\end{lemma}

\begin{proof}
 By Theorem~\ref{regularity} it is enough to prove that $u\in W_{loc}^{1,1}(0, T; X)$.

Consider
$$\psi_{m,n} :=\sup\left\{\inf\left\{\psi,1-\frac{1}{n}\right\},-1+\frac{1}{m}\right\} , \quad m\in \mathbb{N}, \ n,m\ge 2,$$
and let $u_{m,n}$ be the mild solution of problem~\eqref{ProblemI} with Neumann flux $\psi_{m,n}$ and initial data $u_0$. Then,  by Lemma~\ref{nteo317},
\begin{equation}\label{E1nteo317c}\Vert
u_{m,n} \Vert_{L^{\infty}(0, T; L^2(\Omega))} \leq
\left\| u_0\right\|_2+|\Omega|^{1/2}C_\Omega T \quad ,m\in \mathbb{N}, \ n,m\ge 2.
\end{equation}

 For $m$ fixed, we have $\psi_{m,n} \geq -1+\frac{1}{m}$ for all $n\geq 2$. Then, we are under the assumptions of the Step 1 of the proof of Lemma~\ref{general}, hence, by Theorem \ref{Bre-Kom} we have
\begin{equation}\label{E2nteo317c}\Vert
(u_{m,n})_t \Vert_{L^{\infty}(\delta, T; L^2(\Omega))} \leq
\frac{1}{\delta} \Vert u_0 \Vert_2 \ \ \ \ {\rm for}
\ \ 0 < \delta < T,
\end{equation}
and, on the other hand,
$$\lim_{n \to \infty} (I - \widetilde \Delta^{\psi_{m,n,g}}_1)^{-1} f = (I  -\widetilde \Delta^{\psi_m,g}_1)^{-1} f \quad \forall\, f \in L^2(\Omega), $$
being $\psi_m:= \sup\left\{\psi,-1+\frac{1}{m}\right\}$. Then,  by Theorem~\ref{TeoConBP}, we have
\begin{equation}\label{E3nteo317c}
\lim_{n \to \infty} u_{m,n}(t) = u_m(t) \quad \hbox{uniformly on} \ [0,T],
\end{equation}
where $u_m$ is the mild solution of problem~\eqref{ProblemI} with Neumann flux $\psi_{m}$ and initial data $u_0$;
moreover, it satisfies
\begin{equation}\Vert
u_{m} \Vert_{L^{\infty}(0, T; L^2(\Omega))} \leq
\left\| u_0\right\|_2+|\Omega|^{1/2}C_\Omega T \quad ,m\in \mathbb{N}, \ n,m\ge 2.
\end{equation}
and
\begin{equation} \Vert
(u_{m})_t \Vert_{L^{\infty}(\delta, T; L^2(\Omega))} \leq
\frac{1}{\delta} \Vert u_0 \Vert_2 \ \ \ \ {\rm for}
\ \ 0 < \delta < T.
\end{equation}

Now, since $\psi_m$ satisfies the assumption of Step 2 of the proof of  Lemma~\ref{general}, we have
$$\lim_{m \to \infty} (I -\widetilde \Delta^{\psi_m,g}_1)^{-1} f = (I  -\widetilde \Delta^{\psi,g}_1)^{-1} f \quad \forall\, f \in L^2(\Omega).$$
Then, applying again Theorem~\ref{TeoConBP}, we obtain that
\begin{equation}\label{E4nteo317c}
\lim_{m \to \infty} u_{m}(t) = u(t) \quad \hbox{uniformly on} \ [0,T],
\end{equation}
with
\begin{equation} \Vert
u\Vert_{L^{\infty}(0, T; L^2(\Omega))} \leq
\left\| u_0\right\|_2+|\Omega|^{1/2}C_\Omega T \quad ,m\in \mathbb{N}, \ n,m\ge 2.
\end{equation}
and
\begin{equation} \Vert u_t \Vert_{L^{\infty}(\delta, T; L^2(\Omega))} \leq
\frac{1}{\delta} \Vert u_0 \Vert_2 \ \ \ \ {\rm for}
\ \ 0 < \delta < T.
\end{equation}
Then we get that
$$u\in W_{loc}^{1,1}(0, T; L^2(\Omega)).$$
\end{proof}

\section{Large solutions}\label{sectls}

 \begin{proof}[Proof of Theorem \ref{teolsol01}]
  Let $g_{n}:=g\1_{\Gamma_{D,1}}+n\1_{\Gamma_{D,2}}-n\1_{\Gamma_{D,3}}$. Thanks to Theorem~\ref{TeoConBP} it is enough to proof that, for $f\in L^\infty(\Omega)$,
\begin{equation}\label{itse01}\lim_n(I -\widetilde \Delta^{\psi,g_n}_1)^{-1}f =  (I -\widetilde \Delta^{\tilde{\psi},g}_1)^{-1}f.
\end{equation}
   Now,   by Remark~\ref{enlinfbutpost}
  we have that, in fact, for $n$ large enough,
  $$(I -\widetilde \Delta^{\psi,g_n}_1)^{-1}f=
  (I -\widetilde \Delta^{\tilde{\psi},g}_1)^{-1}f,$$
  where $ -\widetilde \Delta^{\tilde{\psi},g}_1$ is the   diffusion operator
associated to breaking the boundary with $\widetilde{\Gamma}_N$ and~$\Gamma_{D,1}$. Then the result follows.
\end{proof}

 In \cite[Example~5.1]{MP} it is shown that if $\Omega = B_1(0)$ in $\R^2$ and the initial datum is the   unbounded function
$$u_0(x):= \left\{ \begin{array}{ll}0 \quad &\hbox{if} \ \ \Vert x\Vert \leq \frac12, \\[10pt] \log \left(\frac{\Vert x\Vert}{1-\Vert x\Vert} \right) \quad &\hbox{if} \ \ \frac12 \leq \Vert x \Vert <1, \end{array} \right.$$
then the solution of problem~\eqref{LargeSolMP} is given by an unbounded function for all time:
$$u(t,x) = a(t) \1_{B_{r(t)}(0)} + \left( \log \left(\frac{\Vert x\Vert}{1-\Vert x\Vert} \right) + \frac{t}{\Vert x\Vert}\right) \1_{B_1(0) \setminus B_{r(t)}(0)},$$
with
$$r(t)= \frac{ W\left(- \frac{t+1}{2e^{t +\frac12}} \right)}{t+1} + 1,$$
where $W$ is the Lambert $W$-function and
$\displaystyle a(t) = \int_0^t \frac{2}{r(s)} \, ds.$ Note that $u(t)$ is not bounded, but $u(t) \in   W^{1,\infty}(\Omega)$.
  With the same technique we are going to get a strong solution $u(t)$ of problem~\eqref{LargeSolMP}, for  $\Omega = B_1(0)$ in $\R^2$, such that $u(t)$ is not in $BV(\Omega)$. This shows that, when $\Vert \psi \Vert_\infty =1$,  the strong solution $u(t)$ of  problem~\eqref{ProblemI}   may not be a $BV$-function, but  $T_k(u(t)) \in BV(\Omega)$ for all $k >0$.

  \begin{example}\label{exlso01}\rm  Let $\Omega = B_1(0)$ in $\R^2$. Take as initial datum the function
$$u_0(x)= \left\{ \begin{array}{ll} 2^{1/4} \quad &\hbox{if} \ \ \Vert x\Vert \leq \frac12, \\[6pt]
\displaystyle \frac{1}{(1-\Vert x\Vert)^{1/4}} \quad &\hbox{if} \ \ \frac12 \leq \Vert x \Vert <1. \end{array} \right.$$
As in \cite[Example~5.1]{MP}, we look for a solution to problem~\eqref{LargeSolMP}  of the form
$$u(t,x) = a(t) \1_{B_{r(t)}(0)} + b\left(t, \Vert x\Vert \right)  \1_{B_1(0) \setminus B_{r(t)}(0)},$$ with   $a(t) = b(t, r(t))$  in $[0, T]$, $b$ increasing in the second variable. Following the same calculations than in \cite[Example~5.1]{MP}  we get that $b(t,r)=\frac{1}{(1-r)^{1/4}} + \frac{t}{r},$ and  $r(t)$  must  solve the   ODE problem
$$\left\{\begin{array}{l}\displaystyle
\left(  \frac{r(t)}{(1- r(t))^{\frac54}} - \frac{t}{r(t)} \right)r'(t) = 1,\\[12pt] \displaystyle r(0) = \frac12,
\end{array}\right.$$
that can be written as a linear ODE in $t(r)$,
\begin{equation}\label{LiODE}
\left\{\begin{array}{l}\displaystyle
\frac{dt}{dr}+\frac{t}{r} =    \frac{r}{(1- r)^{\frac54}},
\\[12pt] \displaystyle t(0) = \frac12,
\end{array}\right.
\end{equation}
whose solution in $[0,+\infty[$ is
$$t(r)=\frac{128 -  2^{1/4}109(1 -  r)^{1/4} - 32 r - 12 r^2}{84 (1 - r)^{1/4} r}.$$ Now, for each $t\in[0,+\infty[$ there is a unique solution $r(t)\in\left[\frac12,1\right[$   of
\begin{equation}\label{lardt}\frac{128 -  2^{1/4}109(1 -  r)^{1/4} - 32 r - 12 r^2}{84 (1 - r)^{1/4} r}=t,
\end{equation}
 it is smooth in $]0,+\infty[$ and its graph is given in Figure~\ref{rdt01}.
\begin{figure}[h]
\centering
    \begin{tikzpicture}
        \begin{axis}[
            axis lines=middle,
            width=7cm, height=5cm,
            ytick={0.5,0.6,0.7, 0.8, 0.9, 1},
            xtick={0.2,0.4,0.6, 0.8,  1},
            samples=1000,
            domain=0:1.1,
            axis line style={-},
            tick label style={font=\tiny}
            ]
            \addplot[black, thick, restrict x to domain=0:1.1]
            ({(128 -  2^(1/4)*109*(1 -  x)^(1/4) - 32*x - 12*x^2)/(84*(1 - x)^(1/4)*x)}, {x});
            \addplot[gray, dashed, restrict x to domain=0:1.1]
            {1};
        \end{axis}
    \end{tikzpicture}
\caption{Graph of $r(t)$}
\label{rdt01}
\end{figure}
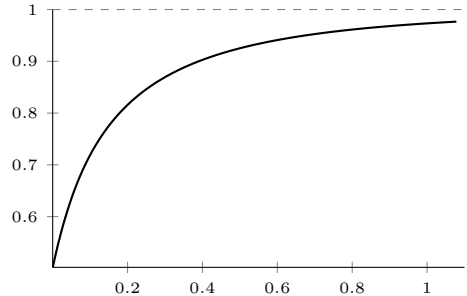
Then,   the solution of problem~\eqref{LargeSolMP} is   given by
$$u(t,x) = a(t) \1_{B_{r(t)}(0)} + \left( \frac{1}{(1-\Vert x\Vert)^{1/4}} + \frac{t}{\Vert x\Vert}\right) \1_{B_1(0) \setminus B_{r(t)}(0)},$$ with
 $r(t)\in\left[\frac12,1\right[$ given by~\eqref{lardt} and with $a(t)= \frac{1}{(1-r(t))^{1/4}} + \frac{t}{r(t)}$.
Since $r(t) < 1$ for all $t\ge \frac12$, and $\frac{1}{(1-\Vert x\Vert)^{1/4}}$ is not in~$BV(\Omega)$, we have that the function $u(t)$ is not in ~$BV(\Omega)$.

\hfill $\blacksquare$
\end{example}

\begin{remark}\rm Large solutions for bounded initial data are indeed bounded. In~\cite{MP} it is shown that, for the solution $u$ of problem~\eqref{LargeSolMP},
\begin{equation}\label{nteo317infp1}\hbox{if $u_0\in L^\infty(\Omega)$ then }    \left\| u(t)\right\|_\infty \le \left\| u_0\right\|_\infty+\frac{N}{s_0} t,\end{equation}
if $\Omega$ satisfies a uniform ball condition with radius $s_0$.
Here we have shown  that, for $\Omega$ with $C^{1,1}$ boundary,
\begin{equation}\label{nteo317infp2}\hbox{if $u_0\in L^\infty(\Omega)$ then }    \left\| u(t)\right\|_\infty \le \left\| u_0\right\|_\infty+C_\Omega t,\end{equation}
being  $C_\Omega$ the constant in Theorem~\ref{pepesab} (observe that $C_\Omega\ge \frac{H^{N-1}(\partial\Omega)}{|\Omega|}$). \hfill $\blacksquare$
\end{remark}

\section{Explicit Solutions}\label{ExplicitS}

In \cite{ABCM1}  it was computed  the solution of the homogeneous Neumann problem for an initial datum given by the characteristic function of a ball $B_r(x_0)$ when $\Omega$ is a ball centered at $x_0$ of radius $R>r$. Now we are going to solve the same case for a non homogenous Neumann boundary condition.

\begin{theorem}\label{Nnohomeneo}
 Consider the problem
\begin{equation}\label{pbej1}
    \left\{
    \begin{array}{ll}
       u_t =  \Delta_1 u   &\hbox{in} \ \ (0,T) \times\Omega,
        \\[10pt] \frac{Du}{\vert Du \vert} \cdot \nu= a \quad &\hbox{in} \ \ (0,T) \times \partial \Omega,
         \\[10pt] u(0) = u_0,
    \end{array}
    \right.
\end{equation}
being $\Omega= B_R(0)$, $u_0= b \1_{B_r(0)}$, with $0 < r < R$, $a, b \in \mathbb{R}$, $\vert a \vert \leq 1$ and $b>0$. Then, the solution of problem \eqref{pbej1} is given by
$$u(t,x) = \left\{ \begin{array}{ll}\displaystyle \left(b - \frac{N}{r}t \right) \,  \1_{B_r(0)}(x) +  N \, \frac{a R^{N-1} + r^{N-1}}{R^N - r^N} \, t \1_{B_R(0) \setminus B_r(0)}(x) \quad &\hbox{if} \ \ 0 \leq t \leq T_1 \\[10pt] \displaystyle \frac{N}{R}a(t-T_1)+\frac{(a R^{N-1} + r^{N-1})r}{R^{N-1}(R+ar)}b  \quad &\hbox{if} \ \  t> T_1,  \end{array} \right.$$
being $\displaystyle T_1=\frac{br(R^N-r^N)}{NR^{N-1}(R+ar)}$.
\end{theorem}

\begin{proof}
Like in the case of homogenous Neumann boundary conditions (see~\cite{ABCM1})  let us see that  we can find a solution of \eqref{pbej1} of the form
$$u(t,x) = \alpha (t) \1_{B_r(0)} + \beta(t) \1_{B_R(0) \setminus B_r(0)}$$
with $\alpha(t) > \beta(t)$ on some interval of time $(0,T_1)$, where $T_1$ is the time at which $\alpha (T_1)=  \beta(T_1)$, and with $\alpha(0) = b$, $\beta(0) =0$. Let us see that we can solve
\begin{equation}\label{E1Ecample}
\alpha'(t) = {\rm div}(\z(t)) \quad \hbox{in} \ \ B_r(0),
\end{equation}
\begin{equation}\label{E2Ecample}
\beta'(t) = {\rm div}(\z(t)) \quad \hbox{in} \ \ B_R(0) \setminus B_r(0),
\end{equation}
\begin{equation}\label{E3Ecample}
[\z(t), \nu_\Omega]= a \quad \hbox{in} \ \ \partial B_R(0),
\end{equation}
for  some $\z(t) \in L^\infty(B_R(0))$, with $\Vert \z(t) \Vert_\infty \leq 1$,   continuous at $\partial B_r(0)$. For $$\z(t)(x) := -\frac{x}{r} \quad \hbox{for} \ x \in B_r(0), $$ integrating the equation  \eqref{E1Ecample} in $B_r(0)$ we obtain
$$\alpha'(t) \mathcal{L}^N(B_r(0)) = \int_{B_r(0)} {\rm div}(\z(t)) dx = \int_{\partial B_r(0)} [\z(t), \nu_\Omega] d\mathcal{H}^{N-1} = - \mathcal{H}^{N-1}(\partial B_r(0)).$$
Thus, $\alpha'(t) = - \frac{N}{r}$, and, therefore $$\alpha(t) = b - \frac{N}{r} \, t.$$

Integrating now \eqref{E2Ecample} in $B_R(0) \setminus B_r(0)$ and having in mind \eqref{E3Ecample}, we get
$$\begin{array}{c}
\displaystyle \beta'(t)  \mathcal{L}^N(B_R(0) \setminus B_r(0)) =  \int_{B_R(0) \setminus B_r(0)} {\rm div}(\z(t)) dx  = \int_{\partial( B_R(0) \setminus B_r(0))} [\z(t), \nu_\Omega] d\mathcal{H}^{N-1}
\\[12pt]
\displaystyle
=  \int_{\partial B_R(0)} [\z(t), \nu_{B_R(0)}] d\mathcal{H}^{N-1} -  \int_{\partial B_r(0)} [\z(t), \nu_{B_r(0)}] d\mathcal{H}^{N-1}
\\[12pt]
\displaystyle= a  \mathcal{H}^{N-1}(\partial B_R(0)) + \mathcal{H}^{N-1}(\partial B_r(0)).
\end{array}$$
Thus,
$$\beta'(t) = \frac{ a  \mathcal{H}^{N-1}(\partial B_R(0)) + \mathcal{H}^{N-1}(\partial B_r(0))}{ \mathcal{L}^N(B_R(0)} = N \, \frac{a R^{N-1} + r^{N-1}}{R^N - r^N}.$$
Therefore
$$\beta(t) = N \, \frac{a R^{N-1} + r^{N-1}}{R^N - r^N} \, t.$$

Note that $T_1$ must be given by
$$
T_1\left( \frac{N}{r}+ N \, \frac{a R^{N-1} + r^{N-1}}{R^N - r^N} \right) = b,
$$
and it is always attained,
$$
T_1=\frac{br(R^N-r^N)}{NR^{N-1}(R+ar)}.
$$

To construct $\z(t)$ in $B_R(0) \setminus B_r(0)$ we shall look for $\z(t)$ of the form
$\z(t)(x) = \rho(\Vert x \Vert) \frac{x}{\Vert x \Vert}$, such that
 $${\rm div}(\z(t)) = \beta'(t) \quad \hbox{in} \ \ B_R(0) \setminus B_r(0),$$
 $$ \rho(r) = -1,  \quad \rho(R) =a,$$   so that it coincides on $\partial B_r(0)$ with the field $\z$ defined on $B_r(0)$. Since
$${\rm div}(\z(t))(x)  = \rho'(\Vert x \Vert) + \rho (\Vert x \Vert) \frac{N-1}{\Vert x \Vert}.$$
Then, we must have
\begin{equation}\label{E6Ecample}
\rho'(s) + \rho(s) \, \frac{N-1}{s} =  N \, \frac{a R^{N-1} + r^{N-1}}{R^N - r^N}, \quad r< s < R.
\end{equation}
Hence
$$\rho(s) =  \frac{a R^{N-1} + r^{N-1}}{R^N - r^N} \,s + C \, \frac{1}{s^{N-1}}.$$
Now, the condition
$\rho(r) = -1$
implies that
$$C=-\frac{r^{N-1}R^{N-1}}{R^N - r^N}(ar +R).$$
Thus,
$$\rho(s) =  \frac{a R^{N-1} + r^{N-1}}{R^N - r^N} \,s -\frac{r^{N-1}R^{N-1}}{R^N - r^N}(ar +R) \, \frac{1}{s^{N-1}}.$$
Note that $\rho(R) = a$ is also satisfied,   and consequently,  we have that~\eqref{E3Ecample} holds.

 Moreover,
 $$ |\rho(s)|\le 1.$$
Indeed, for $N=1$ it is obvious. Let us see it for $N\ge 2$. Observe that if $a>\displaystyle -\left(\frac{r}{R}\right)^{N-1}$ then  both summands defining $\rho$ are increasing, and if $\displaystyle a= -\left(\frac{r}{R}\right)^{N-1}$ then we also have   that $\rho$  is increasing, so, since $\rho(r)=-1$ and $\rho(R)=a\in[-1, 1]$ we have that $-1\le \rho\le 1$. In the case $\displaystyle a<-\left(\frac{r}{R}\right)^{N-1}$ we have that $\rho\ge -1$ if and only if
$$p_1(s):=(a R^{N-1} + r^{N-1})s^N+(R^N-r^N)s^{N-1}-r^{N-1}R^{N-1}(ar+R)\ge 0,$$
and this is true since $p_1(r)=0$, $p_1(R)\ge 0$, $\rho(s)\to -\infty$ as $s\to+\infty$, and this polynomial has an unique critical point different from $0$;
on the other hand, $\rho\le 1$ if and only if
$$p_2(s):=(a R^{N-1} + r^{N-1})s^N-(R^N-r^N)s^{N-1}-r^{N-1}R^{N-1}(ar+R)\le 0,$$
and this is true since $p_2(r)<0$,  $\rho(s)\to -\infty$ as $s\to+\infty$, and the unique critical point of this polynomial different from $0$ is negative.

Consequently, for $t \in (0,T_1)$, the vector field $\z(t)$  given in $B_R(0)$  by
$$\z(t)(x)= \left\{  \begin{array}{ll} - \displaystyle\frac{x}{r}\quad &\hbox{if} \ \ x \in B_r(0), \\[10pt] \displaystyle
\left(\frac{a R^{N-1} + r^{N-1}}{R^N - r^N}\Vert x \Vert  -\frac{R^{N-1}(ar +R)}{R^N - r^N} \, \frac{r^{N-1}}{\Vert x \Vert^{N-1}}\right)\frac{x}{\Vert x \Vert}  &\hbox{if} \ \ x \not\in  B_r(0),\end{array} \right.$$
and
$$u(t,x) = \left(b - \frac{N}{r}t \right) \,  \1_{B_r(0)}(x) +  N \, \frac{a R^{N-1} + r^{N-1}}{R^N - r^N} \, t \1_{B_R(0) \setminus B_r(0)}(x),$$
satisfy~\eqref{E1Ecample}, \eqref{E2Ecample} and~\eqref{E3Ecample}.

From~\eqref{E1Ecample} and~\eqref{E2Ecample}, and the fact that $\z(t)$ is continuous on $\partial B_r(0)$ we get that
\begin{equation}\label{E1}u_t(t,.)={\rm div}(\z(t)) \quad \hbox{in} \ \ B_R(0).
\end{equation}

Let us see now that
\begin{equation}\label{E3}
(\z(t), Du(t)) = |Du(t)| \quad \hbox{as measures}.
\end{equation}
By Proposition \ref{Anzelotti} it is enough to proof
$$\int_{B_R(0)} (\z(t), Du(t)) = \int_{B_R(0)} |Du(t)|.$$
Indeed, applying Green`s formula, we have
$$\int_{B_R(0)} (\z(t), Du(t)) =-\int_{B_R(0)} {\rm div}(\z(t))(x)u(t)(x) \, dx + \beta(t)\int_{\partial B_{R}(0)}[\z(t),\nu_{B_{R}(0)}]d\mathcal{H}^{N-1}.$$
Now
$$-\int_{B_R(0)} {\rm div}(\z(t))(x)u(t)(x) \, dx $$ $$= -\int_{B_r(0)} {\rm div}(\z(t))(x)\alpha(t)(x) \, dx  -\int_{B_R(0) \setminus B_r(0)} {\rm div}(\z(t))(x)\beta(t)(x) \, dx $$ $$=(\beta(t) - \alpha(t))\int_{B_r(0)} {\rm div}(\z(t))(x) \, dx -\int_{B_R(0)} {\rm div}(\z(t))(x)\beta(t)(x) \, dx$$ $$= (\beta(t) - \alpha(t)) \int_{\partial B_r(0)} [\z(t), \nu_{B_r(0)}]  d\mathcal{H}^{N-1} - \beta(t)\int_{\partial( B_R(0)} [\z(t), \nu_{B_R(0)}] ) d\mathcal{H}^{N-1}.$$
Hence
$$\int_{B_R(0)} (\z(t), Du(t)) = (\beta(t) - \alpha(t)) \int_{\partial B_r(0)} [\z(t), \nu_{B_r(0)}]  d\mathcal{H}^{N-1} $$ $$= ( \alpha(t)-\beta(t))\mathcal{H}^{N-1}(\partial B_r(0)) = \int_{B_R(0)} |Du(t)|.$$

 By~\eqref{E3Ecample}, \eqref{E1}, and~\eqref{E3}, we have $u(t,x)$ is a solution of problem \eqref{pbej1} for $0< t \leq T_1$.

  At the time~$T_1$, the solution is flat on the ball $B_R(0)$,
$$u(T_1,x)=\frac{(a R^{N-1} + r^{N-1})r}{R^{N-1}(R+ar)}b.$$
For $t > T_1$,
the solution is given by  $$u(t,x)=\frac{N}{R}a(t-T_1)+\frac{(a R^{N-1} + r^{N-1})r}{R^{N-1}(R+ar)}b.$$
In fact, in this case, it is easy to see that
the vector field $\z(t)$   given by
$$\z(t)(x)= a\frac{x}{R}, \quad x \in B_R(0),$$
satisfies all the conditions of the definition.

\end{proof}

\begin{remark}{\rm
1. Let us give a drawing of how the solution behaves, $u_{|_{B_r(0)}}$ decreases in time linearly and, for $\displaystyle a>-\left(\frac{r}{R}\right)^{N-1}$, $u_{|_{B_R(0) \setminus B_r(0)}}$ increases linearly up to the time $T_1$ where both match, while for $\displaystyle a<-\left(\frac{r}{R}\right)^{N-1}$, $u_{|_{B_R(0) \setminus B_r(0)}}$ also decreases linearly, with slow velocity than $u_{|_{B_r(0)}}$ does, up to the time $T_1$ where both match, and for $\displaystyle a=-\left(\frac{r}{R}\right)^{N-1}$,    $u_{|_{B_R(0) \setminus B_r(0)}}$  stays equal to $0$ up the time $T_1$ where $u_{|_{B_r(0)}}$ reaches the value $0$. At the time~$T_1$, the solution is flat on the ball $B_R(0)$,
$$u(T_1,x)=\frac{(a R^{N-1} + r^{N-1})r}{R^{N-1}(R+ar)}b.$$
From that time on, for $a\neq 0$ the solution is given by an evolving flat $$u(t,x)=\frac{N}{R}a(t-T_1)+\frac{(a R^{N-1} + r^{N-1})r}{R^{N-1}(R+ar)}b,$$
which increases linearly if $a>0$, decreases linearly if $a<0$; and, for the case $a=0$ the solution stays constantly equal to $\displaystyle \frac{r^N}{R^N}b$ (observe that in this case we have mass conservation).

2. Following the same technique it is now easy to see that if $u_0(x)=b\1_{B_R(0) \setminus B_r(0)}(x)$ then the solution of problem~\eqref{pbej1} for $a=1$ is given by
$$u(t,x)=   \frac{N}{r}t  \1_{B_r(0)}(x) +  \left(b+N \, \frac{R^{N-1} - r^{N-1}}{R^N - r^N}t\right)    \1_{B_R(0) \setminus B_r(0)}(x)$$
until the time $t$ at which $$\frac{N}{r}t=b+N \, \frac{R^{N-1} - r^{N-1}}{R^N - r^N}t,$$ observe that $\frac{N}{r}> N \, \frac{R^{N-1} - r^{N-1}}{R^N - r^N}$, from that time on the solution is flat and grows up linearly with velocity $\frac{N}{R}$.
\hfill $\blacksquare$
}
\end{remark}

Let us see now the case of a mixed boundary condition.

\begin{theorem}\label{nMixedE} Consider the problem
\begin{equation}\label{nNeumannProblemNew}
    \left\{
    \begin{array}{ll}
       u_t =  \Delta_1 u   &\hbox{in} \ \ (0,T) \times (B_R(0) \setminus  B_{r_0}(0)),
       \\[10pt] u= 0 \quad &\hbox{on} \ \ (0,T) \times \partial B_{r_0}(0),
        \\[10pt] \frac{Du}{\vert Du \vert} \cdot \nu= 0 \quad &\hbox{on} \ \ (0,T) \times \partial B_R(0),
         \\[10pt] u(0) = u_0,
    \end{array}
    \right.
\end{equation}
being  $u_0= b \1_{B_R(0) \setminus B_r(0)}$, with   $0 < r_0 \leq r < R$   and $b>0$.
 Then, the solution of problem~\eqref{nNeumannProblemNew} is given by:
\\
  (a) in the case $r=r_0$,
$$u(t,x) = \left\{ \begin{array}{l} \left(b - N \frac{r_0^{N-1}}{R^N - r_0^N} \, t \right) \1_{B_R(0) \setminus B_r(0)}(x)   \qquad \hbox{if} \ \ 0 \leq t \leq T_1, \\[10pt] \displaystyle 0   \hfill \hbox{if} \ \  t> T_1,     \end{array} \right.$$
being
$$
T_1=\frac{b}{N}\, \cdot \frac{R^N - r_0^N}{r_0^{N-1}};
$$
(b) in the case $r_0 < r<R$,
$$u(t,x) = \left\{ \begin{array}{l}\displaystyle\left(N\frac{r^{N-1}- r_0^{N-1}}{r^N - r_0^N} \, t \right) \,  \1_{B_r(0)\setminus B_{r_0}}(x) + \left(b - N \frac{r^{N-1}}{R^N - r^N} \, t \right) \1_{B_R(0) \setminus B_r(0)}(x) \\[8pt] \hfill \hbox{if} \ \ 0 \leq t \leq T_1, \\[12pt] \displaystyle \left(b - N \frac{r^{N-1}}{R^N - r^N} \, T_1 \right) -  N\frac{ r_0^{N-1}}{R^N - r_0^N} \, (t -T_1)   \hfill \hbox{if} \ \  t> T_1,     \end{array} \right.$$
being
$$
T_1=\frac{b}{N}\, \cdot \frac{(r^N - r_0^N)(R^N - r^N)}{R^N(r^{N-1} +r_0^{N-1}) -r_0^{N-1}(r^N+r_0^N)}.
$$

\end{theorem}

\begin{proof}   {\it (a)} Suppose that $r_0 = r$, in this case $u_0 = b\1_{B_R(0) \setminus B_{r_0}(0)}$, that is, the initial datum  is a positive constant in $\Omega=B_R(0) \setminus B_{r_0}(0)$. We look for a solution of the form
$$u(t,x) =  \gamma(t) \1_{B_R(0) \setminus B_{r_0}(0)}$$
with $\gamma (0) = b$ and $\gamma(t) > 0$ on some interval of time $(0,T_1)$, where $T_1$ is the time at which $\gamma (T_1)= 0$. Let us see that we can solve
\begin{equation}\label{nE1Ecample20}
\gamma'(t) = {\rm div}(\z(t)) \quad \hbox{in} \ \ B_R(0) \setminus B_{r_0}(0),
\end{equation}
\begin{equation}\label{nE3Ecample2D0}
[\z(t), \nu_{B_R \setminus B_{r_0}}] \in {\rm sign}(-u(t)) = -1 \quad \hbox{in} \ \  \partial B_{r_0}(0),
\end{equation}
\begin{equation}\label{nE3Ecample20}
[\z(t), \nu_{B_R \setminus B_{r_0}}]= 0 \quad \hbox{in} \ \ \partial B_R(0),
\end{equation}
for  some $\z(t) \in L^\infty(B_R(0) \setminus B_{r_0}(0))$, with $\Vert \z(t) \Vert_\infty \leq 1$. Note \eqref{nE3Ecample2D0} and \eqref{nE3Ecample20} means that   $$[\z(t), \nu_{B_{r_0}}] = 1 \ \ \hbox{in} \ \  \partial B_{r_0}(0), \quad \hbox{and} \quad [\z(t), \nu_{B_R}] = 0 \ \  \hbox{in} \ \partial B_R(0).$$
Integrating the equation  \eqref{nE1Ecample20} in $B_R(0) \setminus B_{r_0}(0)$ we obtain
$$\gamma'(t) \mathcal{L}^N(B_R(0) \setminus B_{r_0}(0)) = \int_{B_R(0) \setminus B_{r_0}(0)} {\rm div}(\z(t)) dx = \int_{\partial (B_R(0) \setminus  B_{r_0}(0))} [\z(t), \nu_{B_R \setminus B_{r_0}}]$$ $$= -\int_{\partial B_R(0)} [\z(t), \nu_{B_{r_0}(0)}] d\mathcal{H}^{N-1}  =  -\mathcal{H}^{N-1}(\partial B_{r_0}(0)).$$
Thus, $$\gamma'(t) =-N\frac{ r_0^{N-1}}{R^N - r_0^N},$$ and, therefore $$\gamma(t) =  b - N\frac{ r_0^{N-1}}{R^N - r_0^N} \, t.$$

To construct $\z(t)$ in $B_R(0) \setminus B_{r_0}(0)$ we shall look for $\z(t)$ of the form
$\z(t)(x) = \rho(\Vert x \Vert) \frac{x}{\Vert x \Vert}$, such that $${\rm div}(\z(t)) = \gamma'(t)\quad \hbox{in} \ \ B_R(0) \setminus B_{r_0}(0) ,$$ $$\rho(r_0) = -1, \quad \rho(R) =0.$$

Then, we must have
\begin{equation}\label{nE6Ecample0}
\rho'(s) + \rho(s) \, \frac{N-1}{s} =  -N \frac{r_0^{N-1}}{R^N - r_0^N}, \quad r_0< s < R.
\end{equation}
Hence
$$\rho(s) =  -\frac{ r_0^{N-1}}{R^N - r_0^N} \,s + C \, \frac{1}{s^{N-1}}.$$
Since $\rho(R)=0$ we get
$$C = \frac{r_0^{N-1} R^N}{R^N - r_0^N},$$
so,
$$\rho(s) =  -\frac{ r_0^{N-1}}{R^N - r_0^N} \,s +  \frac{r^{N-1} R^N}{R^N - r_0^N} \, \frac{1}{s^{N-1}},$$
and we have
$$\rho(r_0) =  -\frac{ r_0^{N-1}}{R^N - r_0^N} \,r_0 +  \frac{r^{N-1} R^N}{R^N - r_0^N} \, \frac{1}{r_0^{N-1}} = -1.$$

Consequently, for $t \in (0,T_1)$, the vector field $\z(t)$  given by
$$ \z(t)(x)=   \displaystyle\left(-\frac{ r_0^{N-1}}{R^N - r_0^N} \Vert x \Vert +  \frac{r^{N-1} R^N}{R^N - r_0^N}  \frac{1}{\Vert x \Vert^{N-1}}\right)\frac{x}{\Vert x \Vert}\quad \hbox{if} \ \ x \in B_r(0) \setminus B_{r_0}(0)$$
and
$$u(t,x) =  b - N\frac{ r_0^{N-1}}{R^N - r_0^N} \, t \quad x \in B_R \setminus B_{r_0} $$
satisfy~\eqref{nE1Ecample20}, \eqref{nE3Ecample2D0} and~\eqref{nE3Ecample20}.
 Now
$$u(t,x)=0 \iff t= \frac{b}{N}\, \cdot \frac{(R^N - r_0^N)}{r_0^{N-1}},$$
thus
$$T_1 = \frac{b}{N}\, \cdot \frac{(R^N - r_0^N)}{r_0^{N-1}}.$$

We then have $$u_t(t,.)={\rm div}(\z(t)) \quad \hbox{in} \ \ {B_R(0)\setminus B_{r_0}(0)}.$$
Moreover, since $Du(t) =0$, we have
\begin{equation}\label{nE4Ecample2=}
(\z(t), Du(t)) = \vert Du(t) \vert \quad \hbox{as measures}.
\end{equation}
Therefore $u(t)$ is the solution of problem \eqref{nNeumannProblemNew} for $0 < t < T_1$. Now $u(T_1,x) = 0$ for all $x \in B_R \setminus B_{r_0}$, and consequently, $u(t,x) = 0$ for all $x \in B_R \setminus B_{r_0}$  is the solution of problem \eqref{nNeumannProblemNew} for $t \geq T_1$.

\noindent  {\it (b)}   Suppose now that  $r_0 < r<R$. As in the proof of Theorem \ref{Nnohomeneo}, we look for a solution of the form
$$u(t,x) = \alpha (t) \1_{B_r(0) \setminus B_{r_0}(0)} + \beta(t) \1_{B_R(0) \setminus B_r(0)}$$
with $\alpha(t) < \beta(t)$ on some interval of time $(0,T_1)$, where $T_1$ is the time at which $\alpha (T_1)=  \beta(T_1)$, and with $\alpha(0) = 0$, $\beta(0) =b$. Let us see that we can solve
\begin{equation}\label{nE1Ecample2}
\alpha'(t) = {\rm div}(\z(t)) \quad \hbox{in} \ \ B_r(0) \setminus B_{r_0}(0),
\end{equation}
\begin{equation}\label{nE2Ecample2}
\beta'(t) = {\rm div}(\z(t)) \quad \hbox{in} \ \ B_R(0) \setminus B_r(0),
\end{equation}
\begin{equation}\label{nE3Ecample2D}
  [\z(t), \nu_{B_R \setminus B_{r_0}}]  \in {\rm sign}(-u(t)) \quad \hbox{in} \ \  \partial B_{r_0}(0),
\end{equation}
\begin{equation}\label{nE3Ecample2}
  [\z(t), \nu_{B_R \setminus B_{r_0}}]  = 0 \quad \hbox{in} \ \ \partial B_R(0),
\end{equation}
for  some $\z(t) \in L^\infty( B_R(0) \setminus  B_{r_0}(0))$, with $\Vert \z(t) \Vert_\infty \leq 1$,  continuous at $\partial B_r(0)$. For
$$\z(t)(x) :=\frac{x}{r} \quad \hbox{for} \ x \in \partial B_r(0),$$
and $$-[\z(t), \nu_{B_{r_0}}] = a \in {\rm sign}(-u(t)), $$
integrating the equation  \eqref{nE1Ecample2} in $B_r(0) \setminus B_{r_0}(0)$ we obtain
$$\alpha'(t) \mathcal{L}^N(B_r(0) \setminus B_{r_0}(0)) = \int_{B_r(0) \setminus B_{r_0}(0)} {\rm div}(\z(t)) dx $$ $$= \int_{\partial (B_r(0)\setminus B_{r_0})} [\z(t), \nu_{(B_r(0)\setminus B_{r_0})}] d\mathcal{H}^{N-1} $$ $$=\int_{\partial B_r(0)} [\z(t), \nu_{B_r(0)}] d\mathcal{H}^{N-1} - \int_{\partial B_{r_0}(0)} [\z(t), \nu_{B_{r_0}(0)}] d\mathcal{H}^{N-1} $$ $$= \mathcal{H}^{N-1}(\partial B_r(0)) + a \mathcal{H}^{N-1}(\partial B_{r_0}(0)).$$
Thus, $$\alpha'(t) = N\frac{r^{N-1}+ a r_0^{N-1}}{r^N - r_0^N},$$ and, therefore $$\alpha(t) = N\frac{r^{N-1}+ a r_0^{N-1}}{r^N - r_0^N} \, t.$$
This implies that  $\alpha (t) > 0$, and necessarily $a = -1$, so
$$\alpha(t) = N\frac{r^{N-1}- r_0^{N-1}}{r^N - r_0^N} \, t.$$
Integrating now \eqref{nE2Ecample2} in $B_R(0) \setminus B_r(0)$, we get
$$\beta'(t)  \mathcal{L}^N(B_R(0) \setminus B_r(0)) =  \int_{B_R(0) \setminus B_r(0)} {\rm div}(\z(t)) dx $$ $$ = \int_{\partial( B_R(0) \setminus B_r(0))} [\z(t), \nu_{B_R(0) \setminus B_r(0)}] d\mathcal{H}^{N-1} $$$$=  \int_{\partial B_R(0)} [\z(t), \nu_{B_R(0)}] d\mathcal{H}^{N-1} -  \int_{\partial B_r(0)} [\z(t), \nu_{B_r(0)}] d\mathcal{H}^{N-1} $$ $$= - \mathcal{H}^{N-1}(\partial B_r(0)).$$
Thus,
$$\beta'(t) = -N \frac{r^{N-1}}{R^N - r^N},$$
and therefore
$$\beta(t) = b - N \frac{r^{N-1}}{R^N - r^N} \, t.$$
Note that $T_1$ must be given by
$$
T_1\left(N\frac{r^{N-1}- r_0^{N-1}}{r^N - r_0^N} + N \frac{r^{N-1}}{R^N - r^N} \right) = b,
$$
and it is always attained,
$$
T_1=\frac{b}{N}\, \cdot \frac{(R^N - r^N)(r^N - r_0^N)}{R^N(r^{N-1} -r_0^{N-1}) +r^{N-1}r_0^{N-1}(r -r_0)}.
$$

 Let us now construct $z(t)$ in in $B_R(0) \setminus B_r(0)$ and in $B_r(0) \setminus B_{r_0}(0)$, continuous in $\partial B_r(0)$.

To construct $\z(t)$ in $B_R(0) \setminus B_r(0)$ we shall look for $\z(t)$ of the form
$\z(t)(x) = \rho(\Vert x \Vert) \frac{x}{\Vert x \Vert}$, such that $${\rm div}(\z(t)) = \beta'(t)\quad \hbox{in} \ \ B_R(0) \setminus B_r(0) ,$$ $$\rho(r) = 1, \quad \rho(R) =0.$$
Then, we must have
\begin{equation}\label{nE6Ecample}
\rho'(s) + \rho(s) \, \frac{N-1}{s} =  -N \frac{r^{N-1}}{R^N - r^N}, \quad r< s < R.
\end{equation}
Hence
$$\rho(s) =  -\frac{ r^{N-1}}{R^N - r^N} \,s + C \, \frac{1}{s^{N-1}}.$$
Since $\rho(r)=1$ we get
$$C = \frac{r^{N-1} R^N}{R^N - r^N},$$
so,
$$\rho(s) =  -\frac{ r^{N-1}}{R^N - r^N} \,s +  \frac{r^{N-1} R^N}{R^N - r^N} \, \frac{1}{s^{N-1}}.$$
Note that $\rho(R)=0$ and we have that~\eqref{nE3Ecample2} is indeed true.
 Moreover, since $\rho$ is decreasing in $[r,R]$ and $\rho(r) = 1$ and $\rho(R) =0$, we have $\rho(s) \in [0,1]$ for $s \in [r,R]$ and consequently $\Vert \z(t) \Vert_\infty \leq 1$.

To construct $\z(t)$ in $B_r(0) \setminus B_{r_0}(0)$, again  we   look for $\z(t)$ of the form
$\z(t)(x) = \rho(\Vert x \Vert) \frac{x}{\Vert x \Vert}$, but now such that $${\rm div}(\z(t)) = \alpha'(t)\quad \hbox{in} \ \ B_r(0) \setminus B_{r_0}(0),$$ $$\rho(r) = 1, $$
so that it coincides on $\partial B_r(0)$ with the field $\z(t)$  obtained on $B_R(0)\setminus B_r(0)$,  and with
$$  \rho(r_0) =1.$$
Then, we must have
\begin{equation}\label{nE6EcampleN}
\rho'(s) + \rho(s) \, \frac{N-1}{s} =   N\frac{r^{N-1}- r_0^{N-1}}{r^N - r_0^N}, \quad r_0< s <r.
\end{equation}
Hence
$$\rho(s) = \frac{r^{N-1}- r_0^{N-1}}{r^N - r_0^N} \,s + C \, \frac{1}{s^{N-1}}.$$
Now, $\rho(r)=1$ implies that
$$C = \left(1 - \frac{r^{N-1}-r_0^{N-1}}{r^N - r_0^N} \,r\right) r^{N-1} = \frac{ r - r_0}{r^N - r_0^N} \, r_0^{N-1} \,r^{N-1},$$
so
$$\rho(s) = \frac{r^{N-1}- r_0^{N-1}}{r^N - r_0^N} \,s + \frac{ r - r_0}{r^N - r_0^N} \, r_0^{N-1} \,r^{N-1} \, \frac{1}{s^{N-1}}.$$
Now, $\rho(r_0) =1$ and then~\eqref{nE3Ecample2D} is indeed satisfied.
  Moreover,  it is easy to see that $ |\rho(s)|\le 1.$

Consequently, for $t \in (0,T_1)$, the vector field $\z(t)$  given by
$$ \z(t)(x)= \left\{  \begin{array}{l}  \displaystyle\left(\frac{r^{N-1}- r_0^{N-1}}{r^N - r_0^N} \Vert x \Vert + \frac{ (r - r_0)r_0^{N-1} r^{N-1} }{r^N - r_0^N} \, \frac{1}{\Vert x \Vert^{N-1}}\right)\frac{x}{\Vert x \Vert}\\[14pt]\hfill\hbox{if} \ \ x \in B_r(0) \setminus B_{r_0}(0), \\[10pt]  \displaystyle\left(-\frac{ r^{N-1}}{R^N - r^N} \Vert x \Vert +  \frac{r^{N-1} R^N}{R^N - r^N} \, \frac{1}{\Vert x \Vert^{N-1}}\right)\frac{x}{\Vert x \Vert}  \qquad \hfill\hbox{if} \ \ x \in B_R(0) \setminus B_r(0),\end{array} \right.$$
and
$$u(t,x) = \left(N\frac{r^{N-1}- r_0^{N-1}}{r^N - r_0^N} \, t \right) \,  \1_{B_r(0)\setminus B_{r_0}}(x) + \left(b - N \frac{r^{N-1}}{R^N - r^N} \, t \right) \1_{B_R(0) \setminus B_r(0)}(x)$$
satisfy~\eqref{nE1Ecample2}, \eqref{nE2Ecample2}, \eqref{nE3Ecample2D} and~\eqref{nE3Ecample2}.

From~\eqref{nE1Ecample2} and~\eqref{nE2Ecample2}, and the fact that $\z(t)$ is continuous on $\partial B_r(0)$ we get that
\begin{equation}\label{nE1}u_t(t,.)={\rm div}(\z(t)) \quad \hbox{in} \ \ {B_R(0)\setminus B_{r_0}(0)}.
\end{equation}
Let us see now that \begin{equation}\label{nE4Ecample2}
(\z(t), Du(t)) = \vert Du(t) \vert \quad \hbox{as measures}.
\end{equation}
  By Proposition \ref{Anzelotti} it is enough to proof
 $$\int_{B_R(0)\setminus B_{r_0}(0)} (\z(t), Du(t)) = \int_{B_R(0)\setminus B_{r_0}(0)} |Du(t)|.$$
 Now, this follows since $\z$ is continuous at $\partial B_{r}(0)$. Indeed, applying Green`s formula, we have
$$\int_{B_R(0)\setminus B_{r_0}(0)} (\z(t), Du(t)) $$ $$=-\int_{B_R(0)\setminus B_{r_0}(0)}{\rm div}(\z(t))(x)u(x) \, dx
+\int_{\partial B_{R}(0)}[\z(t),\nu_{B_{R}(0)}]u -\int_{\partial B_{r_0}(0)}[\z(t),\nu_{B_{r_0}(0)}]u$$ $$= -\int_{B_R(0)\setminus B_{r_0}(0)}{\rm div}(\z(t))(x)u(x) \, dx+ \alpha(t) \mathcal{H}^{N-1}(\partial B_{r_0}(0) ).$$
On the other hand,
$$-\int_{B_R(0)\setminus B_{r_0}(0)} {\rm div}(\z(t))(x)u(x) \, dx $$ $$=-\int_{B_R(0)\setminus B_{r}(0)} {\rm div}(\z(t))(x)u(x)\, dx
-\int_{B_r(0)\setminus B_{r_0}(0)} {\rm div}(\z(t))(x)u(x).$$
Now,
$$-\int_{B_R(0)\setminus B_{r}(0)} {\rm div}(\z(t))(x)u(x) \, dx $$ $$=\int_{B_R(0)\setminus B_{r}(0)}(\z(t),Du)-
\int_{\partial B_R(0)}[\z(t),\nu_{B_R(0)}]u+ \int_{\partial B_r(0)}[\z(t),\nu_{B_r(0)}]u$$
$$=\beta(t)\mathcal{H}^{N-1}( \partial B_r(0) ) $$
and
$$-\int_{B_r(0)\setminus B_{r_0}(0)} {\rm div}(\z(t))(x)u(x) \, dx $$ $$=\int_{B_r(0)\setminus B_{r_0}(0)}(\z(t),Du)-
\int_{\partial B_r(0)}[\z(t),\nu_{B_r(0)}]u+ \int_{\partial B_{r_0}(0)}[\z(t),\nu_{B_{r_0}}]u$$
$$=-\alpha(t)\mathcal{H}^{N-1}( \partial B_r(0) )-\alpha(t)\mathcal{H}^{N-1}(\partial B_{r_0}(0) ) .$$
Therefore
$$\int_{B_R(0)\setminus B_{r_0}(0)} (\z(t), Du(t)) =(\beta(t)-\alpha(t))\mathcal{H}^{N-1}(\partial B_r(0)) = \int_{B_R(0)\setminus B_{r_0}(0)}|Du| .$$

We have that at $t=T_1$,
  $$u(T_1,x) = \left(b - N \frac{r^{N-1}}{R^N - r^N} \, T_1 \right) \quad \hbox{for all} \ x \in B_R(0)\setminus B_{r_0}(0). $$
Then, by (a), we have
$$u(t,x) =  \left(b - N \frac{r^{N-1}}{R^N - r^N} \, T_1 \right) -  N\frac{ r_0^{N-1}}{R^N - r_0^N} \, (t -T_1)$$
 is the solution of problem \eqref{nNeumannProblemNew} for $t \geq T_1$.
\end{proof}

\appendix
\section{Reminder on some basic  tools in nonlinear semigroup theory}

 We collect in this Appendix some   results concerning the nonlinear semigroup Theory. For more details one can consult  \cite{BCrP}, \cite{Barbu}, \cite{CrandallLiggett} or \cite{BrezisBook}.

Let $X$ be a Banach space,     $W_{loc}^{1,1}(0, T; X)$ denotes the space of all locally absolutely continuous functions $u: [0,T] \rightarrow X$ which  are differentiable almost every where on $[0,T]$.  We have that $u \in W^{1,1}(0, T; X)$  if and only if there exists a function  $g \in L^1(0,T; X)$ such that
$$u(t) = u(a) + \int_a^t g(s) ds \quad \hbox{for} \   a,t \in [0,T],$$
and then $u'(t) = g(t)$ almost every where. When the Banach space $X$ has {\it the Radom-Nikodym property}, for instance when $X$ is reflexive, then absolutely continuous functions are differentiable almost every where.

Let $A : X \rightarrow 2^X$ be an operator and consider the {\it abstract Cauchy problem}
\begin{equation}\label{CP}
	\left\{ \begin{array}{ll}
		u^{\prime}(t) + Au(t) \ni 0 \ \ \ \ \ {\rm on} \ \ t \in (0,
		T), \\[10pt] u(0) = x.
	\end{array} \right.
\end{equation}

\begin{definition}\rm A function $u$ is called a {\it strong solution} of problem~\eqref{CP} if
	$$
	\left\{ \begin{array}{lll}
		u \in C([0, T]; X) \cap W_{loc}^{1,1}(0, T; X), \\[10pt] u^{\prime} +
		Au(t) \ni 0 \ \ {\rm a.e.} \ t
		\in (0, T),
		\\[10pt] u(0) = x.
	\end{array} \right.
	$$
\end{definition}

\begin{definition}{\rm  Let $\varepsilon > 0$. An $\varepsilon$-{\it
			discretization} of $$u^{\prime}+ Au \ni 0$$ on $[0,T]$ consists of a partition $t_0 < t_1 < \cdots < t_N$  such that,
		$$t_i - t_{i-1} \leq \varepsilon, \ i = 1,
		\ldots , N, \ \ t_0 \leq \varepsilon \ \ {\rm and} \ \ T - t_N \leq \varepsilon.
		$$
		
		We will denote this discretization by $D_A(t_0, \ldots , t_N)$.

		A {\it solution of the discretization} $D_A(t_0, \ldots , t_N$ is a piecewise
		constant function $v : [t_0, t_N] \rightarrow X$ whose values $v(t_0) = v_0$, $v(t) = v_i$ for $t
		\in ]t_{i-1}, t_i]$, $i = 1, \ldots , N$ satisfy
		\begin{equation}\label{4E3}
			\frac{v_i - v_{i-1}}{t_i - t_{i-1}} + Av_i \ni 0, \ \ \ \ i = 1, \ldots , N.
		\end{equation}
		
		A {\it mild solution} of problem \eqref{CP} is a continuous function $u \in C([0, T]; X)$ such
		that, $v(0) = x$ and for each $\varepsilon > 0$ there is $D_A(t_0, \ldots , t_N)$ an $\varepsilon$-discretization of $u^{\prime}
		+ Au \ni 0$ on $[0,T]$ which has a solution $v$ satisfying
		$$\Vert u(t) - v(t) \Vert \leq \varepsilon \ \ \ \ \ {\rm for} \ \ t_0 \leq t \leq t_N.$$
	}
\end{definition}

It is well known that {\it every strong solution is a mild solution}. the reciprocal, in general, is not true

In order to have uniqueness of mild solutions we need to introduce the following class of operators.

\begin{definition} {\rm An operator $A$ in $X$ is {\it accretive} if
		$$\Vert x - \hat{x} \Vert \leq \Vert x - \hat{x} + \lambda (y - \hat{y}) \Vert \ \ \ {\rm
			whenever} \ \ \lambda > 0 \ \ {\rm and} \ \ (x, y), (\hat{x},
		\hat{y}) \in A.$$
		That is,   $A$ is accretive if and only if $(I + \lambda A)^{-1}$ is a singlevalued nonexpansive
		map for $\lambda \geq 0$.
	}
\end{definition}

\begin{definition}\label{maccretivity} {\rm An operator $A$ is called m-{\it
			accretive}\index{m-accretive operator} in
		$X$ if and only if $A$ is accretive and $R(I + \lambda A) = X$ for all $\lambda > 0$.
	}
\end{definition}

We have the following existence and uniqueness result.

\begin{theorem}\label{EUm-accretive}  Let $A$ be an operator in $X$ and $x_0
	\in \overline{D(A)}$. If $A$ is m-accretive, then the problem
	$$u^{\prime} + A u \ni f \ \ {\rm on} \ \ [0, T], \ \ \ u(0) = x_0$$
	has a unique mild solution $u$ on  $[0, T]$.   Moreover we have the Crandall-Ligett's exponential formula
	\begin{equation}\label{CRLExp}
		u(t) = \lim_{n \to \infty}\left( I + \frac{t}{n} A \right)^{-n} u_0.
	\end{equation}
\end{theorem}

In general every strong solution is a mild solution. We have the following regularity result.
\begin{theorem}\label{regularity} Suppose that $A$ is an $m$-accretive operator in $X$ and $u$ is a mild solution of
	$$u^{\prime} + A u \ni f \ \ {\rm on} \ \ [0, T], \ \ \ u(0) = x_0.$$
	
	If $u \in  W_{loc}^{1,1}(0, T; X)$  then $u$ is a strong solution.
	
	If $X$ has the Radom-Nikodym property and $x \in D(A)$ then   $u$ is a strong solution
\end{theorem}

In the particular case that the Banach space is a Hilbert space $(H, ( \  | \ ))$ be a Hilbert space the accretivity of an operator $A$ is equivalent to its monotonia, i.e., $A$ is accretive if and only if $A$ is {\it monotone} in the sense that
\begin{equation}\label{3E4}
	(x - \hat{x} | y - \hat{y}) \geq 0 \ \ \ \ \ {\rm for \ all} \ \ (x, y), (\hat{x}, \hat{y}) \in
	A.
\end{equation}

In the hilbertian framework, we have the following result.

\begin{theorem}({\bf Minty's Theorem.})\label{Minty} Let $H$ be a Hilbert space and $A$
	an accretive operator in $H$. Then, $A$ is m-accretive if and only if $A$ is maximal monotone.
\end{theorem}

We have the following existence and uniqueness result \cite[Th\'{e}or\`{e}me 3.4]{BrezisBook}.

\begin{theorem}\label{Breisresult} Let $H$ be a Hilbert space and $A$ a maximal monotone operator in $H$  and $u_0 \in \overline{D(A)}$, then the
	mild solution $u(t)$ of
	$$
	\left\{
	\begin{array}{l}
		u^{\prime} +
		A(u) \ni 0 \ \ \mbox{ on } \ [0, T], \\[10pt]
		u(0) = u_0,
	\end{array}
	\right.
	$$
	is a  weak solution in the sense of \cite[Definition 3.1]{BrezisBook}.
\end{theorem}

One of the more important class of maximal monotone in Hilbert spaces are the subdifferential of convex lower-semicontinuous functionals in Hilbert spaces. We remember that for a proper functional $\varphi : H
\rightarrow (- \infty, +\infty]$, that is
$D(\varphi) := \{ x  \in H \ : \ \varphi (x) \not= + \infty \} \not= \emptyset$, its  {\it subdifferential operator} $\partial \varphi$  is defined by
$$w \in \partial \varphi (z) \ \Longleftrightarrow \ \varphi (x) \geq \varphi(z) + (w | x - z) \
\
\ \ \forall \, x \in H.$$
For such operators we have the following regularity.

\begin{theorem}{\bf (Brezis-Komura Theorem)}\label{Bre-Kom}
	Let H be a Hilbert space and $\varphi : H \rightarrow (- \infty, +\infty]$ a  proper, convex and lower
	semi-continuous function and $u_0 \in \overline{D(\partial \varphi)}$, then the
	mild solution $u(t)$ of
	$$
	\left\{
	\begin{array}{l}
		u^{\prime} +
		\partial
		\varphi(u) \ni 0 \ \ \mbox{ on } \ [0, T], \\[10pt]
		u(0) = u_0,
	\end{array}
	\right.
	$$
	is a strong solution  and we have the following estimate
	$$\Vert
	u^{\prime} \Vert_{L^{\infty}(\delta, T; H)} \leq
	\frac{1}{\delta} \Vert u_0 \Vert \ \ \ \ {\rm for}
	\
	\ 0 < \delta < T.
	$$
\end{theorem}

We have the following interesting convergence result.

\begin{theorem}\label{TeoConBP} {\bf (Brezis-Pazy Theorem)}
	Let $A_n$ be $m$-accretive in $X$, $x_n \in
	\overline{D(A_n)}$ and $f_n \in L^1(0, T; X)$ for $n = 1, 2,
	\ldots , \infty$. Let $u_n$ be the mild solution of
	$$u'_n + A_n u_n \ni f_n, \ \ \hbox{in} \ [0,T], \ \ u_n(0) = x_n.$$
	If $f_n \to f_{\infty}$ in $L^1(0, T; X)$ and  $x_n \to x_{\infty}$ as $n \to \infty$ and
	$$\lim_{n \to \infty} (I + \lambda A_n)^{-1} z = (I + \lambda A_{\infty})^{-1} z,$$
	for some $\lambda > 0$ and all $z \in D$, with $D$ dense in $X$, then
	$$\lim_{n \to \infty} u_n(t) = u_{\infty}(t) \ \ \hbox{uniformly on} \ [0,T].$$
\end{theorem}

Let us also collect some preliminaries and notations concerning
completely accretive operators that will be used afterwards (see
\cite{BCr2}). Let $(\Sigma, \mathcal{B}, \mu)$ be a $\sigma$-finite
measure space, and $M(\Sigma,\mu)$ the space of $\mu$-a.e. equivalent
classes of measurable functions $u : \Sigma\to \R$.
Let
\begin{displaymath}
	J_0:= \Big\{ j : \R \rightarrow
	[0,+\infty] : \text{$j$ is convex, lower
		semicontinuous, }j(0) = 0 \Big\}.
\end{displaymath}
For every $u$, $v\in M(\Sigma,\mu)$, we write
\begin{displaymath}
	u\ll v \quad \text{if and only if}\quad \int_{\Sigma} j(u)
	\,d\mu \le \int_{\Sigma} j(v) \, d\mu\quad\text{for all $j\in J_{0}$.}
\end{displaymath}

\begin{definition}{\rm
		An operator $A$ on $M(\Sigma,\mu)$ is called {\it completely
			accretive} if for every $\lambda>0$ and
		for every $(u_1,v_1)$, $(u_{2},v_{2}) \in A$ and $\lambda >0$, one
		has that
		\begin{displaymath}
			u_1 - u_2 \ll u_1 - u_2 + \lambda (v_1 - v_2).
		\end{displaymath}
		If $X$ is a linear subspace of $M(\Sigma,\mu)$ and $A$ an operator
		on $X$, then $A$ is {\it $m$-completely accretive on $X$} if $A$ is
		completely accretive and satisfies the {\it range
			condition}
		\begin{displaymath}
			\textrm{Ran}(I+\lambda A)=X\qquad\text{for some (or equivalently, for
				all) $\lambda>0$.}
		\end{displaymath}
	}
\end{definition}

We denote
\begin{displaymath}
	L_0(\Sigma,\mu):= \left\{ u \in M(\Sigma,\mu) \ : \
	\int_{\Sigma} \big[\vert u \vert
	- k\big]^+\, d\mu < \infty \text{ for all $k > 0$} \right\}.
\end{displaymath}
The following results were proved in \cite{BCr2}.

\begin{proposition}
	\label{prop:completely-accretive}
	Let $P_{0}$ denote the set of all functions $q\in C^{\infty}(\R)$
	satisfying $0\le q'\le 1$,
	$q'$  is compactly supported, and $0$ is not contained in the
	support ${\rm supp}(q)$ of $q$. Then,
	an operator $A \subseteq L_{0}(\Sigma,\mu)\times
	L_{0}(\Sigma,\mu)$ is  completely accretive if and only if
	\begin{displaymath}
		\int_{\Sigma}q(u-\hat{u})(v-\hat{v})\, d\mu \ge 0
	\end{displaymath}
	for every $q\in P_{0}$ and every $(u,v)$, $(\hat{u},\hat{v})\in A$.
\end{proposition}

{\flushleft \bf Acknowledgements.}      The first author would like to acknowledge the CNRST of Morocco for their partial support through the Fincom program.     The second and third authors  have been partially supported  by  Grant PID2022-136589NB-I00 funded by MCIN/AEI/10.13039/501100011 033 and FEDER and by Grant RED2022-134784-T funded by MCIN/AEI/10.13039/501 100011033.    The authors also thank the EST Essaouira for its hospitality and welcome in Essaouira during the CIMPA School 2025, where part of the results and the writing was finalized.

For the purpose of open access, the authors have applied a CC BY public copyright licence to any Author Accepted Manuscript version arising from this submission.

\

{\bf Data Availability.} Data sharing is not applicable to this article as no datasets were generated or
analyzed during the current study.

\

{\bf Conflict of interest.} The authors have no conflict of interest to declare that are relevant to the content of this
article.


\begin{thebibliography}{IMNT}

\bibitem{AF} R. A. Adams and J. J. F. Fornier, {\it Sobolev Spaces}, Academic Press 2003.

\bibitem{AFP}
L. Ambrosio, N. Fusco and D. Pallara,
\newblock {\it Functions of Bounded Variation
and Free Discontinuity Problems},
\newblock Oxford Mathematical Monographs, 2000.

\bibitem{ABCM1}
F. Andreu, C. Ballester, V. Caselles and J. M. Maz\'on,
\newblock {\it Minimizing Total Variation Flow},
\newblock  Diff. and Int. Eq. {\bf 14} (2001), 321--360.

\bibitem{ABCM2}
F. Andreu, C. Ballester, V. Caselles and J. M. Maz\'on,
\newblock {\it The Dirichlet problem for the total variational flow},
\newblock  J. Funct. Anal. {\bf 180} (2001), 347--403.


\bibitem{ACMBook} F.~Andreu, V.~Caselles, and J. M. Maz\'on, {Parabolic Quasilinear Equations Minimizing Linear Growth Functionals}, Progress in Mathematics, vol. 223, Birkh\"auser, Basel, 2004.





\bibitem{Anzellotti1}
G. Anzellotti,
\newblock {\it Pairings between measures and bounded functions
and compensated compactness},
\newblock Ann. di Matematica Pura ed Appl. IV (135)
(1983), 293--318.


\bibitem{Anzellotti2}
G. Anzellotti, {\it The Euler equation for functionals with linear growth},
 Trans. Amer. Math. Soc. {\bf 290}  (1985), 483--500.

 \bibitem{BBCS} C. Ballester, , M. Bertalmio, V. Caselles,  G. Sapiro and J. Verdera,{\it  Filling-in by joint interpolation of vector fields and gray levels}. IEEE Trans. Image Process. {\bf 10} (2001), no. 8, 1200--1211.

    \bibitem{Barbu} V. Barbu, {\it Nonlinear differential equations of monotone types in Banach spaces}
Springer Monogr. Math.
Springer, New York, 2010.

\bibitem{BCN} G. Bellettini, V. Caselles and M. Novaga, {\it Explicit solutions of the eigenvalue problem  $-\divi \left(\frac{Du}{\vert Du \vert}\right)=u$  in  $\R^2$},
SIAM J. Math. Anal. {\bf 36} (2005),  1095--1129.


\bibitem{BCr2}
Ph. B\'{e}nilan and M. G. Crandall,
\newblock {\it Completely Accretive Operators},
\newblock in Semigroups Theory and Evolution Equations, Ph. Clement et
al. editors, Marcel Dekker, 1991, pp. 41--76.

\bibitem{BCrP}
Ph. B\'enilan, M. G. Crandall and A. Pazy,
\newblock {\it Evolution Equations Governed by Accretive Operators},
\newblock Book in preparation.


  \bibitem{BL} J. Bergh and J. L\"ofstrom, Interpolation Spaces. An Introduction, Grundlehren der mathematischen Wissenschaften {\bf 223}, Springer, Berlin, 1976.



\bibitem{BrezisBook}
 H.~Br\'ezis, {\it Op\'erateurs maximaux monotones et semi-groupes de
  contractions dans les espaces de {H}ilbert}, North-Holland Publishing Co.,
  Amsterdam-London; American Elsevier Publishing Co., Inc., New York, 1973.
\newblock North-Holland Mathematics Studies, No. 5.




\bibitem{BrPa}   H. Brezis and
A . Pazy, \newblock {\it Convergence and approximation of semigroups of nonlinear operators in Banach spaces. }
\newblock  J. Funct. Anal. {\bf 9} (1972), 63--74.



\bibitem{CrandallLiggett}
M. G. Crandall and T. M. Liggett,
\newblock {\it Generation of Semigroups of Nonlinear
Transformations on General Banach Spaces},
\newblock Amer. J. Math., {\bf 93}, (1971), 265--298.




\bibitem{DemengelTeman1} F. Demengel and R. Temam, {\it Convex functions of a measure and applications}, Indiana Univ. Math. J. {\bf 33} (1984), 673--709.	
	
\bibitem{ET}  I. Ekeland and R. Teman, Convex Analysis and Variational Problems, North-Holland,
Amsterdam, 1976.

\bibitem{Giaetal} M. Giaquinta, G. Modica and J. Sou$\check{\rm c}$ek, {\it Functionals witdh linear growth in the calculus of variations  II}, Math. Univ. Carolinae {\bf 20} (1979), 143--156.

\bibitem{GMBook}  W. G\'{o}rny and J. M. Maz\'{o}n, Functions of Least Gradient, Monographs in Mathematics, vol. 110, Birkh\"{a}user, 2024.


%
\bibitem{Mazya}
 V.~Maz'ya, {\it Sobolev spaces with applications to elliptic partial
  differential equations}, vol.~342 of Grundlehren der Mathematischen
  Wissenschaften [Fundamental Principles of Mathematical Sciences], Springer,
  Heidelberg, augmented~ed., 2011.


\bibitem{Modica}
 L.~Modica, {\it Gradient theory of phase transitions with boundary
  contact energy}, Ann. Inst. H. Poincar\'e Anal. Non Lin\'eaire, 4 (1987)
  pp.~487--512.


\bibitem{MP} S. Moll and F. Petitta, {\it Large solutions for nonlinear parabolic equations without absorption terms,} Journal Functional Analysis  {\bf  262} (2012), 1566--1602.


  \bibitem{RS} M. Reed and B. Simon, {\it  Methods of Modern Mathematical Physics,} Vol. 1, Functional Analysis, Section III.3. Academic Press, San Diego, 1980.

      \bibitem{ROF} L. Rudin, S. Osher and E. Fatemi, {\it Nonlinear total variation based noise removal algorithms}, Physica D. {\bf 60} (1992), 259--268.

\bibitem{Sabina} J. C. Sabina de Lis, {\it The trace inequality for $BV(\Omega)$  in smooth domain}, preprint \texttt{https://arxiv.org/abs/2411.17325}.

\bibitem{SWZ} P. Sternberg, G. Williams and W. P. Ziemer, {\it Existence, uniqueness, and regularity for functions of least gradient}, J. Reine Angew. Math. {\bf 430} (1992), 35--60.



\end{thebibliography}
\end{document}